\documentclass[a4paper,10pt]{article}

\usepackage{graphicx}
\usepackage{float}
\usepackage{hyperref}

\usepackage{amsmath,amsthm,amssymb}
\usepackage{enumerate}
\usepackage{color}
\usepackage{mathrsfs}
\usepackage{underscore}
\allowdisplaybreaks[4]

\renewcommand{\theequation}{\thesection.\arabic{equation}}

\newcommand{\Rmnum}[1]{\expandafter\@slowromancap\romannumeral #1@}

\makeatother

\title{Global boundedness and asymptotic behavior of the chemotaxis system for Alopecia Areata with weakly singular sensitivity}

\author {Pengxue Xiang$^{a}$,\quad Yuebo Cao$^{a}$,\quad Hongying Yang$^{a,}$\thanks {Corresponding author: yanghongying2000@shzu.edu.cn.}\\
\small  $^a$ School of Sciences, Shihezi University, Shihezi 832003, P.R. China
 }
\date{}
\newtheorem{theorem}{Theorem}[section]

\newtheorem{lemma}{Lemma}[section]

\theoremstyle{definition}
\newtheorem{remark}{Remark}[section]
\begin{document}
\baselineskip14pt \maketitle
\renewcommand{\theequation}{\arabic{section}.\arabic{equation}}
\catcode`@=11 \@addtoreset{equation}{section} \catcode`@=12
\noindent{\bf Abstract} 

This paper considers the homogeneous Neumann initial-boundary value problem for Alopecia Areata chemotaxis model with weakly singular sensitivity. For any appropriately regular initial conditions, it is shown that the problem admits a global boundedness of classical solutions in two spatial dimensions. Moreover, through the explicit construction of Lyapunov functions, we establish that the globally bounded solution converges exponentially to a constant steady state. The paper concludes with numerical experiments that serve to visually illustrate and corroborate some of the theoretically derived findings.

\vspace{\baselineskip} 
\noindent{\bf MSC}: 92C17; 92C30; 35K65; 35B45; 35B40

\vspace{\baselineskip} 
\noindent{\bf Keywords}: Chemotaxis; Alopecia areata; Weakly singular sensitivity; Boundedness; Asymptotic behavior

\section{ Introduction}
\hspace*{\parindent}
Alopecia areata (AA) is a common non-scarring hair loss disorder, typically characterized by sudden patchy hair loss on the scalp in round or oval shapes. In severe cases, hair loss may involve the entire scalp or even the whole body. Although the disease is not life-threatening, it can easily lead to psychological issues such as image-related anxiety and depression. Its pathogenesis is complex, primarily involving the erroneous attack of hair follicles by T lymphocytes, which forces hair follicles into the telogen phase. This process involves the interaction among CD4\(^+\) T cells, CD8\(^+\) T cells and interferon-gamma (IFN-\(\gamma\)). IFN-\(\gamma\) is secreted by both CD8\(^+\) and CD4\(^+\) T cells, diffuses and gradually attenuates in the body, while T cells exhibit both random and directed migration characteristics\cite{gilhar2012alopecia,luster1987biochemical}.
In recent years, research on simulating the occurrence and development process of such types of alopecia using mathematical modeling methods has gradually drawn attention. Dobreva et al. \cite{dobreva2020toward} proposed 
\begin{eqnarray}{\label{eq:1.1}}
\left\{
\begin{array}{llll}
u_t = \Delta u -  \chi_1\nabla \cdot (u \nabla w) +w - \mu_1 u^2, & x \in \Omega, t > 0, \\
v_t = \Delta v - \chi_2\nabla \cdot (v \nabla w) +w+ruv -\mu_2v^2, & x \in \Omega,  t > 0,\\
w_t=\Delta w+u+v-w,& x \in \Omega,  t > 0
\end{array}
\right.
\end{eqnarray}
for the unknown functions \(u=u(x,t)\), \(v=v(x,t)\) and \(w=w(x,t)\) denote the densities of CD4\(^+\) T cells, CD8\(^+\) T cells and the concentration of IFN-\(\gamma\), respectively, where \(\Omega\subset \mathbb{R}^n\) \((n\geq1)\) is a bounded domain with smooth boundary, and \(\chi_i\), \(\mu_i\)(\(i=1\), \(2\)), \(r\) are positive parameters. Currently, there have been some studies on the system \eqref{eq:1.1}. For instance, when \(n = 2\), Lou and Tao \cite{lou2021role} demonstrated that the system (\ref{eq:1.1}) admits a globally bounded classical solution for any \(\mu_1\), \(\mu_2>0\), and when \(n = 3\), they established the global boundedness of solutions under the conditions that \(\mu_1 > 8\chi_1^2 + \frac{r}{2} + 16\), \(\mu_2 > 8\chi_2^2 + \frac{r}{2} + 16\) and \(\mu_1\mu_2^2 > \frac{4}{27}r^3\).  Besides, they proved that when
\(\mu_1<\mu_2<3\mu_1\) and \(n=2\) or \(\mu_1<\mu_2<3\mu_1-32\) with \(\mu_1>16\), \(\mu_2>16\) and \(n=3\), if there exists \(\chi_0>0\) such that \(\sqrt{\chi_1^2+\chi_2^2}\leq \chi_0\), then the solutions of system \eqref{eq:1.1} converge to  \( ( \frac{2}{\mu_1}\), \(\frac{2}{\mu_1}\), \(\frac{4}{\mu_1} ) \) as \( t \to \infty \). Subsequently, Zhang et al. \cite{zhang2023global} extended the findings of \cite{lou2021role} to \(n=4\) and $n=5$, assuming sufficiently large values for \(\mu_i(i = 1, 2)\). When the sensitivity function satisfies \(\chi_i(w) \leq K_i(1 + \alpha_1 w)^{-k_i}\)(\( k_i > 1 \),\(i = 1, 2 \)) and \(n \geq 1\), Qiu et al. \cite{qiu2024boundedness} demonstrated that the system \eqref{eq:1.1} possesses a globally bounded classical solution and this solution converges to a constant steady state. When there are singularities in the chemotactic sensitivity function \(\chi(w)=\frac{\chi_i}{w}\)\((i=1,2)\) and \(n=2\), Gao and Xu \cite{gao2024global} established that the system (\ref{eq:1.1}) admits a globally bounded classical solution as long as \( \chi_1\), \(\chi_2 < \frac{\sqrt{5}}{2} \) and this solution convergs to a constant steady state. Later, Tu et al.\cite{tu2024boundedness} showed under the condition \( \mu_1 \geq 2r\), \(\mu_2 \geq 2r\), \(\text{max}\{\chi_1\), \(\chi_2\} < \sqrt\frac{2}{n} \), the system \eqref{eq:1.1} possesses a globally bounded classical solution, which extends the previous boundedness result from \( n = 2 \) to \( n \geq 2 \). When the system \eqref{eq:1.1} takes into account quasilinear diffusion \(D_1(u)\geq u^m\), \(D_2(v)\geq v^l\)(\(m,l\in \mathbb{R}\)) and \(n\geq 1\), Xu et al. \cite{xu2023boundedness}  obtained that the system \eqref{eq:1.1} possesses a globally bounded classical solution under specific conditions \(\mu_2 \geq r\) and \(m\), \(l \geq 1-\frac{2}{n}\). When  \(D_i(s)\geq(s+1)^{\alpha_i}\), nonlinear chemotaxis \(S_i(s)\geq s(s+1)^{\beta_i-1}\), nonlinear logistic source and \(n\geq 1\), Zhang et al.\cite{zhang2024boundedness} asserted that the system \eqref{eq:1.1} admits a global bounded classical solutions without imposing any further restrictions on parameters. When the diffusion rate \(d_i(i=1,2,3)\) are constants, \(\chi_i=\frac{1}{(1+w)^2}\) and $n=1$ or $2$, Zhang and Fu \cite{zhang2025global} showed that the system \eqref{eq:1.1} has a globally bounded classical solution and this solution converges exponentially to the steady state if \(\mu_1 < \mu_2 < 3\mu_1-2\), \(r = \mu_2 - \mu_1\) and \(\frac{1}{16d_3}(\frac{\chi^2_1(0)}{d_1}+\frac{\chi^2_2(0)}{d_2})<\frac{1}{L^2_w}\). Subsequently, when \(\chi_i=\frac{1}{(1+w)^2}\), \(n=2\) and \(r<\text{min}\{3\mu_1\), \(\frac{5}{4}\mu_2\}\), they obtained that the system \eqref{eq:1.1} has a globally bounded classic solution if  transformed the diffusion rates into a function  \(d_i(w)\)\((i=1,2)\), while \(d_3\) remains a constant. Also, they verified the asymptotic behavior of the solution through numerical simulation\cite{zhang2025globalexp}. When the chemotactic sensitivity function possesses singularities \(\chi_i(w)=\frac{1}{w}\),  \(D_i(s)\geq(s+1)^{\alpha_i}\), \(-\mu_1 u^{\gamma_1}\),\(-\mu_2 v^{\gamma_2}\) (\(\gamma_i>2\)) (\(i=1,2\)) and \(n=2\), Zhou et al.\cite{zhou2024boundedness} proved that the system \eqref{eq:1.1} has a globally bounded classical solution and this solution converges to a constant steady state.

For the corresponding parabolic-parabolic-elliptic variant of \eqref{eq:1.1}, Tao and Xu \cite{tao2022combined}  proved that the system \eqref{eq:1.1} has a global and bounded classical solution under the conditions \(\mu_1 - \frac{r}{2} > \frac{(n-2)_+}{n} \left(2\chi_1 + \frac{\chi_2}{2}\right)\) and \(\mu_2 - r > \frac{(n-2)_+}{n} \left(2\chi_2 + \frac{\chi_1}{2}\right)\) for \(n \geq 1\). Moreover, under the conditions \(\mu_1 < \mu_2 < 3\mu_1\), \(r = \mu_2 - \mu_1\) and \(\chi_1^2 + \chi_2^2 < 2\mu_1(3\mu_1 - \mu_2)\), they demonstrated that any global solution converges asymptotically to a constant equilibrium. Later, when \(n\geq 1\), Shan and Zheng \cite{shan2023boundedness} explored the same variant with \(D_i(s)\geq (s+1)^{\alpha_i}\), \(0\leq S(s)\leq s^{\beta_i}\), \(-\mu_1 u^{\eta_i}\), \(-\mu_2 v^{\eta_i}\) (\(\alpha_i \in \mathbb{R}\), \(\beta_i\geq 0\), \(i=1,2\)), the system \eqref{eq:1.1} exists globally bounded classical solutions and this solution converges exponentially to the steady state if  \(\chi\) is appropriately mild. For more the global solutions and large time behavior of the system \eqref{eq:1.1}, the reader can refer to \cite{shan2023stability,song2023spatiotemporal,zhang2024uniform,shan2025global}.

Currently, no research has been conducted on the system \eqref{eq:1.1} with weakly singularity. To gain a better understanding of weakly singularity, we first review the classical Keller-Segel model with weakly singularity
that is, \(\chi(v)=\frac{\chi}{v^k}\), \(k\in(0,1)\), \(\chi>0\), as exemplified by the following chemotaxis model
\begin{eqnarray}{\label{eq:1.2}}
\left\{
\begin{array}{llll}
u_t = \Delta u - \chi\nabla \cdot (\frac{u}{v^k} \nabla v) +ru-\mu u^2, & x \in \Omega, t > 0, \\
\tau v_t = \Delta v -\alpha v+\beta u , & x \in \Omega,  t > 0.
\end{array}
\right.
\end{eqnarray}
When \(\tau=0\) and \(n=2\), Zhao \cite{zhao2023boundedness} proved that the system \eqref{eq:1.2} has a globally bounded classical solution if \(\mu\) is suitably large, without establishing a uniformly positive lower bound for \(v\). when \(n\geq3\),  Kurt \cite{kurt2024large} showed that the \eqref{eq:1.2} has a globally bounded classical solution and this solution exponentially converges to the constant steady state \((\frac{r}{\mu},\frac{\beta}{\alpha},\frac{r}{\mu})\). When\(n\geq 2\), he proved that the system \eqref{eq:1.2} admits a globally uniformly bounded classical solution under the additional assumption \(k<\frac{1}{2}+\frac{1}{n}\)   \cite{kurt2025boundedness}. Later, Le and Kurt\cite{le2025global} established that the system possesses a globally bounded classical solution provided that \(\mu\) is sufficiently large, with the required threshold depending on the value of \(n\). When we consider that the system \eqref{eq:1.2} is fully parabolic, that is, \(\tau=1\) and \(n=2\), Le \cite{le2025absence} proved that, by considering an energy function that allows the singularity of \(v\) near zero to be incorporated into the diffusion effect, the system \eqref{eq:1.2} admits uniformly bounded solutions. Subsequently, when \(\mu\) is sufficiently large, he demonstrated that the system \eqref{eq:1.2} admits a unique nonnegative classical solution, which extends the previous boundedness result from \(n = 2\) to \(n\geq 3\)\cite{le2025globalweak}.

In view of the aforementioned analysis of the classical Keller- Segel model with weakly singularity, in this paper, we consider the following model \eqref{eq:1.3}
\begin{eqnarray}{\label{eq:1.3}}
\left\{
\begin{array}{llll}
u_t = \Delta u - \chi_1\nabla \cdot (\frac{u}{w^k} \nabla w) +w - \mu_1 u^2, & x \in \Omega, t > 0, \\
v_t = \Delta v - \chi_2\nabla \cdot (\frac{v}{w^k} \nabla w) +w+ruv -\mu_2v^2, & x \in \Omega,  t > 0,\\
w_t=\Delta w+u+v-w,& x \in \Omega,  t > 0,\\
\displaystyle \frac{\partial u}{\partial \nu} = \frac{\partial v}{\partial \nu} = \frac{\partial w}{\partial \nu} = 0, & x \in \partial \Omega, \ t > 0, \\
u(x, 0) = u_0(x), \ v(x, 0) = v_0(x), \ w(x, 0) = w_0(x), & x \in \Omega
\end{array}
\right.
\end{eqnarray}
in a bounded domain \(\Omega \subset \mathbb{R}^2\) with a smooth boundary \(\partial\Omega\), where \(\nu\) denotes the outward unit normal vector to \(\partial\Omega\). The parameters \(\chi_i, \, \mu_i \, (i = 1, 2)\), \(r\) are positive and \(k \in (0,1)\). And the initial data \((u_0, v_0, w_0)\) satisfies
\begin{eqnarray}\label{eq:1.4}
\left\{
\begin{array}{llll}
u_0 \in C^0(\bar{\Omega}), \text{ with } u_0 \geq 0 \text{ and } u_0 \neq 0 \text{ in } \bar{\Omega}, \\ 
v_0 \in C^0(\bar{\Omega}), \text{ with } v_0 \geq 0 \text{ in } \bar{\Omega}, \\ 
w_0 \in W^{1,\infty}(\Omega), \text{ with } w_0 \geq 0 \text{ in } \bar{\Omega}.   
\end{array}
\right.
\end{eqnarray} 

We mainly study the globally bounded classical solutions of the system \eqref{eq:1.3} and their asymptotic behavior. To this end, we present the proof outline of this article:

\noindent\textbf{Main ideas.}  We adopt the analytical framework in \cite{le2025absence}, introducing an energy functional 
\begin{align*}
y(t) = \int_{\Omega} u \ln u+\int_{\Omega} v \ln v - \lambda \int_{\Omega} u \ln w - \lambda \int_{\Omega} v \ln w + \frac{1}{2} \int_{\Omega} |\nabla w|^2.    
\end{align*}
Using this, we obtain an \( L \ln L \) bound for both  \( u \) and \( v \) without relying on a lower bound for \( w \).
However, transitioning from this bound to \( L^p \) bounds for \( p > 1 \) presents another challenge: \( w \) can still be arbitrarily close to 0. To address this, we direct our attention towards the energy function for further analysis
\begin{align*}
z(t) = \int_{\Omega} u^p w^{-q}+\int_{\Omega} v^p w^{-q} + \int_{\Omega} u^p +\int_{\Omega} v^p+ \int_{\Omega} |\nabla w|^{2p},    
\end{align*}
where \( 0 < q < p - 1 \), which allows us to absorb the singularity of \( w \) near 0 into the diffusion.
On the other hand, the asymptotic behavior of the solution to \eqref{eq:1.3} can be derived by an analogous estimate in \cite{lou2021role}. The key to achieving this is to establish the following functional
\begin{align*}
\mathcal{F}(t) &:= \int_{\Omega} \left\{ u(\cdot, t) - u_* - u_* \ln \frac{u(\cdot, t)}{u_*} \right\} + \int_{\Omega} \left\{ v(\cdot, t) - v_* - v_* \ln \frac{v(\cdot, t)}{v_*} \right\}\\
&\quad+ 2 \int_{\Omega} \left\{ w(\cdot, t) - w_* - w_* \ln \frac{w(\cdot, t)}{w_*} \right\} \quad \text{for all } t > 0,    
\end{align*}
which satisfies an energy estimate 
\[\frac{d}{dt}\mathcal{F}(t) \leq -\varepsilon_1 \mathcal{E}(t)\]
with \(\mathcal{E}(t) \) defined in Lemma \ref{lem:4.1}. Moreover, we draw on the idea from \cite{2019Jinhaiyang}. We  introduce a function \(\varphi(\omega) := \omega- u_* \ln \omega\) for \(\omega > 0\) and apply L'Hôpital's rule to obtain 
\begin{align*}
\lim_{\omega \to u_*} \frac{\varphi(\omega) - \varphi(u_*)}{(\omega - u_*)^2} = \frac{1}{2u_*}.  
\end{align*}
Through simple analysis and calculation, we have
\begin{align*}
\frac{d}{dt} \mathcal{F}(t) \leq -\varepsilon_2\mathcal{F}(t) \quad \text{for all} \quad t > t_3.   \end{align*}
where \(t_3\) defined in Lemma \eqref{lem:4.4}. Then, using the Gronwall's inequality and Gagliardo–Nirenberg inequality, we can show that the solution of the system \eqref{eq:1.3} converges exponentially to the steady state.
Then, we present the following two theorems below:
\begin{theorem}\label{thm:1.1}
Let \( \Omega \subset \mathbb{R}^2 \) be a bounded domain with smooth boundary and \( \mu_1, \mu_2 , r > 0 \). Assume that \( \chi_i>0 \) (\( i = 1, 2 \)), \(k \in (0,1)\). Assume \(u_0\), \(v_0\) and \(w_0\) satisfy \eqref{eq:1.4}. The system \eqref{eq:1.3} possesses a unique global classical solution \( (u, v, w) \) which is bounded in \(\Omega\times(0,\infty)\), that is, there exists a constant \( C > 0 \) such that  
\[
\|u(\cdot, t)\|_{L^\infty(\Omega)} + \|v(\cdot, t)\|_{L^\infty(\Omega)} + \|w(\cdot, t)\|_{W^{1,\infty}(\Omega)} \leq C \quad \text{for all} \quad t > 0.
\]
\end{theorem}
\begin{remark}
\text We extend the method previously developed for single-species, single-chemical signaling (see \cite{le2025absence}) to two-species, single-chemical signaling systems, which requires additional handling of the coupling term \(ruv\).
\end{remark}
\begin{remark}
Compared with \cite{gao2024global}, this paper deals with the weakly singular sensitivity case, which is more challenging.
\end{remark}
The next theorem is devoted to showing the large time behavior of solutions, the solution of (\ref{eq:1.3}) converges to a steady state \((u_*, v_*, w_*)\) as \(t \to \infty\), which is given by
\begin{align}\label{eq:1.5}
u_* := \frac{1 + a}{\mu_1}, \quad v_* := a u_* \quad \text{and} \quad w_* := \mu_1 u_*^2,    
\end{align}
where
\begin{align}\label{eq:1.6}
a := \frac{r + \sqrt{r^2 + 4\mu_1 \mu_2}}{2\mu_2}.    
\end{align}
Then, we can obtain the following result based on Theorem \ref{thm:1.1}.
\begin{theorem}\label{thm:1.2} 
Let ($u$, $v$, $w$) be the solution obtained in 
Theorem \ref{thm:1.1}. Then we have the following results: 

\noindent
(1) If \( \mu_1 < \mu_2 < 3\mu_1 \) and \( r := \mu_2 - \mu_1 \), then the global bounded solution \( (u, v, w) \) of \eqref{eq:1.3} converges to the steady state, that is,
\begin{align*}
\|u(\cdot, t) - u_*\|_{L^\infty(\Omega)} + \|v(\cdot, t) - v_*\|_{L^\infty(\Omega)} + \|w(\cdot, t) - w_*\|_{L^\infty(\Omega)} \to 0 \quad \text{as} \quad t \to \infty,    
\end{align*}
where \( u_* = v_* = \frac{2}{\mu_1} \) and \( w_* = \frac{4}{\mu_1} \). 

\noindent
(2)Assume the conditions in Theorem \ref{thm:1.2}(1) hold. And \(\chi_i(w)=\frac{\chi_i}{w^k} \) (\( i = 1, 2 \)) satisfy the following assumptions:
\hspace*{\parindent}\item[\bf(H1)]
\begin{itemize}
 \( \chi_i(w) \in C^2(0, \infty) \) with \(  \chi_i(w) > 0 \) for all \( w > 0 \);
\end{itemize}
\item[\bf(H2)]  
\begin{itemize}
\( \lim_{w \to +\infty} \chi_i(w)\) exist.
\end{itemize}
\noindent Then there exist positive constants \(C\) and \(\lambda\) independent of \(t\) such that any positive solutions of system \eqref{eq:1.3} satisfy
\begin{align*}
\|u(\cdot, t) - u_*\|_{L^\infty(\Omega)} + \|v(\cdot, t) - v_*\|_{L^\infty(\Omega)} + \|w(\cdot, t) - w_*\|_{L^\infty(\Omega)} \leq C e^{-\lambda t}, \quad t > 0
\end{align*}
with \( u_* = v_* = \frac{2}{\mu_1} \) and \( w_* = \frac{4}{\mu_1} \).    
 
\end{theorem}

\begin{remark}
Due to the presence of weakly singularity in our model compared to the one in \cite{lou2021role}, the scaling process becomes more complex.    
\end{remark}
\begin{remark}
We conducted numerical simulations to verify the conclusions of Theorem \ref{thm:1.2}.    
\end{remark}

Without confusion, we will use \(c_i\) \((i=1,2,...)\) to denote generic positive constants which may vary in the context, and the integration variables \(x\) and \(t\) will be omitted, for instance, \(\int_{0}^{t}\int_{\Omega} f(x,s)dxds\) will be abbreviated as \(\int_{0}^{t}\int_{\Omega} f\).
\section{Preliminaries}
\hspace*{\parindent}
In this section, we present the local existence and uniqueness of solutions as well as some basic estimates. The first lemma, which states that there is the local existence and uniqueness of this classical solution, has a proof based on the standard contraction mapping argument and the fixed-point argument in Banach spaces.
\begin{lemma}\textup{\cite{winkler2010boundedness,horstmann2005boundedness}}\label{lem:2.1}
Let \(\Omega \subset \mathbb{R}^2 \) be a smoothly bounded domain. Assume that the functions \(\chi_i > 0\) \((i = 1, 2)\), and the nonnegative initial data \((u_0, v_0, w_0)\) satisfies \eqref{eq:1.4}. Then there exist \(T_{\text{max}} \in (0, \infty]\) and a uniquely determined triple \((u, v, w)\) of positive functions from \(C^0(\bar{\Omega} \times [0, T_{\text{max}})) \cap C^{2,1}(\bar{\Omega} \times (0, T_{\text{max}}))\) solving \eqref{eq:1.3} in \(\Omega \times (0, T_{\text{max}})\). In addition,
\begin{align*}
\text{either} \quad T_{\text{max}} = \infty, \quad \text{or} \quad \lim_{t \nearrow T_{\text{max}}} \sup \left\{ \|u(\cdot, t)\|_{L^\infty(\Omega)} + \|v(\cdot, t)\|_{L^\infty(\Omega)} \right\} = \infty.     
\end{align*}
\end{lemma}  

From this point onward, we shall refer to \((u,v,w)\) as the distinct solution to the system \eqref{eq:1.3} within the domain \(\Omega \times (0, T_{\text{max}})\), where \(T_{\text{max}} \in (0, \infty]\), as confirmed by Lemma \ref{lem:2.1}. The upcoming lemma, which constitutes a parabolic regularity theorem within Sobolev spaces, will be utilized subsequently in Lemma \ref{lem:3.2} and during the demonstration of Theorem \ref{thm:1.1}. The proof is grounded in the standard \( L^p-L^q \) estimates pertaining to the Neumann heat kernel; for an in-depth demonstration.

\begin{lemma}\textup{\cite{horstmann2005boundedness}[Lemma 4.1]}\label{lem:2.2}
Let \(\Omega \subset \mathbb{R}^2\) be an open bounded domain with smooth boundary. Assume that \(p \geq 1\) and \(q \geq 1\) satisfy
\[
\begin{cases} 
q < \frac{2p}{2-p}, & \text{when } p < 2, \\ 
q < \infty, & \text{when } p = 2, \\ 
q = \infty, & \text{when } p > 2.
\end{cases}
\]
Assuming \(W_0 \in W^{1,q}(\Omega)\) and \(W\) is a classical solution to the following system
\[
\begin{cases} 
W_t = \Delta W - W + f & \text{in } \Omega \times (0, T), \\ 
\dfrac{\partial W}{\partial \nu} = 0 & \text{on } \partial\Omega \times (0, T), \\ 
W(\cdot, 0) = W_0 & \text{in } \Omega,
\end{cases} 
\]
where \(T \in (0, \infty]\). If \(f \in L^\infty((0, T)\); \(L^p(\Omega))\), then \(W \in L^\infty((0, T)\); \(W^{1,q}(\Omega))\).
\end{lemma}

The subsequent lemma presents a valuable differential inequality that will be subsequently utilized in Lemma \ref{lem:3.3}.

\begin{lemma}\textup{\cite{le2025absence}[Lemma 2.3]}\label{lem:2.3}
Assume that nonnegative function \( y \in C^1([0, T]) \) \begin{align*}
\int_t^{t+\tau} y(s) \, ds \leq L_1 \quad \text{for all } t \in (0, T-\tau),    
\end{align*}
where \( \tau = \min \{1, \frac{T}{2}\} \), \( L_1 > 0 \). And
\begin{align*}
y'(t) \leq h(t)y(t) + g(t) \quad \text{for all } t \in (0, T),    
\end{align*}
where \( h \) and \( g \) are nonnegative continuous functions in \([0, T)\) such that
\begin{align*}
\int_t^{t+\tau} h(s) \, ds \leq L_2 \quad \text{and} \quad \int_t^{t+\tau} g(s) \, ds \leq L_3 \quad \text{for all } t \in (0, T-\tau),   
\end{align*}
where \( L_2 > 0 \) and \( L_3 > 0 \). 
Then there exists \( C > 0 \) such that \( y(t) \leq C \) for all \( t \in (0, T) \).  
\end{lemma}

The following lemma, directly deduced from \cite{winkler2024logarithmically} [Corollary 1.2], presents a refined variant of the Gagliardo–Nirenberg interpolation inequality. This refined inequality will subsequently serve as a crucial tool for deriving 
 \( L^p \) bounds for both \( u \) and \( v \)
in Lemma \ref{lem:3.7}. For the proof of this lemma, we refer to \cite{le2025boundedness} [Lemma 2.6] and [Lemma 2.7].
\begin{lemma}\label{lem:2.4}
Let \(\Omega \subset \mathbb{R}^2\) be a bounded domain with a smooth boundary, and suppose \(m > 0\) and \(\gamma > \xi \geq 0\). Then, for every \(\varepsilon > 0\), there exists a constant \(C  > 0\) such that the inequality
\begin{align*}
\int_{\Omega} \phi^{m+1} \ln^{\xi}(\phi + e) \, dx \leq &\varepsilon \left( \int_{\Omega} \phi \ln^{\gamma}(\phi + e) \, dx \right) \left( \int_{\Omega} |\nabla \phi^{\frac{m}{2}}|^2 \, dx \right)\\
&+ \varepsilon \left( \int_{\Omega} \phi \, dx \right)^m \left( \int_{\Omega} \phi \ln^{\gamma}(\phi + e) \, dx \right) + C   
\end{align*}
holds for any nonnegative function \(\phi \in C^1(\Omega)\).   
\end{lemma}
\begin{proof} 
Since \( \gamma > \xi \geq 0 \), one can verify that for any \( \delta > 0 \), there exists \( c_1 > 0 \) such that for any \( a \geq 0 \) we have
\[
a^{m+1} \ln^\xi(a + e) \leq \delta a^{m+1} \ln^\gamma(a + e) + c_1. 
\]
This entails that
\begin{align*}
\int_\Omega \phi^{m+1} \ln^\xi(\phi + e) \leq \delta \int_\Omega \phi^{m+1} \ln^\gamma(\phi + e) + c_1 |\Omega|.    
\end{align*}
By applying Sobolev's inequality when \( n = 2 \), there exists a positive constant \( c_2 \) such that
\begin{align}\label{eq:2.1}
\int_{\Omega} \phi^{m+1} \ln^\gamma (\phi + e) \leq c_2 \left( \int_{\Omega} \left| \nabla \left( \phi^{\frac{m+1}{2}} \ln^\frac{\gamma}{2} (\phi + e) \right)  \right|  \right) ^2 + c_2 \left( \int_{\Omega} \phi \ln^\frac{\gamma}{m+1} (\phi + e) \right)^{m+1}.    
\end{align}
By using elementary inequalities, one can verify that
\begin{align*}
\left| \nabla \left( \phi^{\frac{m+1}{2}} \ln^\frac{\gamma}{2} (\phi + e) \right) \right| \leq c_3 \phi^{\frac{1}{2}} \ln^\frac{\gamma}{2} (\phi + e) |\nabla \phi^{\frac{m}{2}}| ,   
\end{align*}
where \( c_3 = C(m, \gamma) > 0 \). This, together with Hölder's inequality leads to
\begin{align}\label{eq:2.2}
c_2 \left( \int_{\Omega} \left| \nabla\left( \phi^{\frac{m+1}{2}} \ln^\frac{\gamma}{2} (\phi + e) \right)\right|\right)^2 \leq c_4 \int_{\Omega} \left| \nabla \phi^{\frac{m}{2}}  \right|^2\cdot \int_{\Omega}\phi \ln^\gamma (\phi + e),    
\end{align}
where \( c_4 = c_2 c_3 \). By Hölder's inequality, we deduce that
\begin{align}\label{eq:2.3}
c_2 \left( \int_{\Omega} \phi \ln^\frac{\gamma}{m+1} (\phi + e) \right)^{m+1} \leq c_2 \left( \int_{\Omega} \phi \right)^m \left( \int_{\Omega} \phi \ln^\gamma (\phi + e) \right).     
\end{align}
Collecting \eqref{eq:2.1} - \eqref{eq:2.3} implies that
\begin{equation}\label{eq:2.4}
 \int_{\Omega} \phi^{m+1} \ln^\gamma (\phi + e)\leq C \int_{\Omega} \left| \nabla \phi^{\frac{m}{2}}  \right|^2\cdot \int_{\Omega}\phi \ln^\gamma (\phi + e)+ C\left( \int_{\Omega} \phi \right)^m \left( \int_{\Omega} \phi \ln^\gamma (\phi + e) \right).   
\end{equation}
Now for any fixed \( \varepsilon \), we choose \( \delta = \frac{\varepsilon }{C} \)  and apply \eqref{eq:2.4} to have the desire inequality in this lemma. 
\end{proof}
\section{Global boundedness}
\hspace*{\parindent}
In this section, we present several significant estimates for solutions, particularly deriving the  \( L \ln L \) boundedness and \( L^p \) boundedness for  \( u \) and \( v \). And on this basis, we will prove Theorem 1.1. Let's commence with an 
\( L^1 \)  bound as outlined in the following lemma:
\begin{lemma}\textup{\cite{lou2021role}[Lemma 2.2]}{\label{lem:3.1}}
There exist \( M> 0 \) and \( L > 0 \) such that the solution of \eqref{eq:1.3} satisfies
\begin{align}{\label{eq:3.1}}
\int_{\Omega} u(\cdot, t) \leq M  \quad \int_{\Omega} v(\cdot, t) \leq M \quad \text{and} \quad \int_{\Omega} w(\cdot, t) \leq M \quad \text{for all } t \in (0, T_{\max}),    
\end{align}  
as well as
\begin{align}{\label{eq:3.2}}
\int_{t}^{t+\tau} \int_{\Omega} u^2 \leq L \quad \text{and} \quad \int_{t}^{t+\tau} \int_{\Omega} v^2 \leq L \quad \text{for all } t \in (0, T_{\max} - \tau),   
\end{align}    
where \[\tau := \min \left\{ 1, \frac{1}{4} T_{\max} \right\}.\]  
\end{lemma} 

With an established \( L^1 \) bound for both \( u \) and \( v \) , we can proceed to derive \( L^p \) bounds for \( w \) for any \( p \geq 1 \). This is made possible by the parabolic regularity result in Sobolev spaces, as demonstrated in Lemma \ref{lem:2.2}.
\begin{lemma} \label{lem:3.2}
For any \( p \geq 1 \), there exists \( C  > 0 \) such that  
\[
\int_\Omega w^p(\cdot, t) \leq C \quad \text{for all } t \in (0, T_{\text{max}}).
\]    
\end{lemma}
\begin{proof}
Integrating the third equation of  \eqref{eq:1.3} over \( \Omega \) and applying Lemma \ref{lem:3.1} yields  
\begin{align*}
\frac{d}{dt} \int_\Omega w + \int_\Omega w &= \int_\Omega \Delta w+ \int_\Omega u +\int_\Omega v\\&=-\int_\Omega |\nabla w|^2+ \int_\Omega u +\int_\Omega v\leq 2 M.     
\end{align*} 
Therefore, applying  Gronwall's inequality to this implies that  
\begin{align*}
 \sup_{t \in (0, T_{\text{max}})} \int_\Omega w(\cdot, t) \leq \max \left\{ \int_\Omega w_0, 2M\ \right\}.   
\end{align*} 
Since \( u,v \in L^\infty ( (0, T_{\text{max}})\); \(L^1(\Omega))  \), we apply Lemma \ref{lem:2.2}  to  obtain  that \( w \in L^\infty ( (0, T_{\text{max}})\); \(W^{1, q}(\Omega) \) for any \( q \in [1, 2) \). Now, applying Sobolev's inequality deduces that  
\[
\int_\Omega w^p \leq c_1 \left( \int_\Omega |\nabla w|^{\frac{2p}{2 + p}} \right)^{\frac{p + 2}{2}} + c_1 \left( \int_\Omega w \right)^p \leq c_2 \quad \text{for all } t \in (0, T_{\text{max}}),  
\]  
where \( c_1 > 0 \) and \( c_2 > 0 \). The proof is now complete.  
\end{proof}  
The lemma below offers an \( L^2 \) bound for the gradient of 
 \( w \), a result that will be subsequently employed in Lemma \ref{lem:3.5}.
\begin{lemma}\label{lem:3.3} 
There exists \( C > 0 \) such that  
\[
\int_{\Omega} |\nabla w(\cdot, t)|^2 \leq C  \quad \text{for all } t \in (0, T_{\text{max}}).
\]
\end{lemma}
\begin{proof}
Multiplying the third equation of \eqref{eq:1.3} by \( w \) and applying Young's inequality yields  
\begin{align}{\label{eq:3.3}}
\frac{1}{2} \frac{d}{dt} \int_{\Omega} w^2 &= -\int_{\Omega} |\nabla w|^2 + \int_{\Omega} uw + \int_{\Omega} vw-\int_{\Omega}w^2 \notag\\&\leq -\int_{\Omega} |\nabla w|^2 + \frac{1}{2} \int_{\Omega} u^2+ \frac{1}{2} \int_{\Omega} v^2.    
\end{align}
For all \( t \in (0, T_{{max}}-\tau)\), integrate both sides of \eqref{eq:3.3} from \(t\) to \(t+\tau\). 
This, together with \eqref{eq:3.2} and Lemma \ref{lem:3.2}  leads to  
\begin{align}{\label{eq:3.4}}
&\int_{t}^{t+\tau} \int_{\Omega} |\nabla w(\cdot, s)|^2 \,\mathrm{d}x\mathrm{d}s \notag\\
\leq& \frac{1}{2} \int_{t}^{t+\tau} \int_{\Omega} u^2(\cdot, s) \,\mathrm{d}x\mathrm{d}s +\frac{1}{2} \int_{t}^{t+\tau} \int_{\Omega} v^2(\cdot, s) \,\mathrm{d}x\mathrm{d}s+ \frac{1}{2} \int_{\Omega} w^2(\cdot, t) \leq c_1.    
\end{align}
Let \( y(t) := \frac{1}{2} \int_{\Omega} |\nabla w(\cdot, t)|^2 \). Through simple calculations, we have 
\begin{align*}
y'(t) + 2y(t) &=\int_{\Omega} \nabla w\cdot \nabla w_t+\int_{\Omega} |\nabla w(\cdot, t)|^2 \\&=-\int_{\Omega}w_t \Delta w+\int_{\Omega} |\nabla w(\cdot, t)|^2\\&= -\int_{\Omega} (\Delta w)^2 - \int_{\Omega} u \Delta w- \int_{\Omega} v \Delta w\\ 
&\leq \frac{1}{2} \int_{\Omega} u^2+\frac{1}{2} \int_{\Omega} v^2.
\end{align*}  
This, combine with \eqref{eq:3.2} - \eqref{eq:3.4} and Lemma \ref{lem:2.3} entails that \( \sup_{t \in (0, T_{\text{max}})} y(t) \leq c_2 \) for some \( c_2 > 0 \), which completes the proof. 
\end{proof}
 
The lemma below is simple but essential for deriving an  \( L \ln L \) estimate of \( u \) and \( v \).

\begin{lemma}\label{lem:3.4}
Let \( l(t):= -\int_{\Omega} u \ln w  -\int_{\Omega} v \ln w \). Then
\begin{align*}
\varepsilon l'(t) + \varepsilon l(t) \leq  & 2 \varepsilon \int_{\Omega} \frac{|\nabla u|^2}{u}+ 2\varepsilon \int_{\Omega} \frac{|\nabla v|^2}{v}- \frac{\varepsilon}{2} \int_{\Omega} u \frac{|\nabla w|^2}{w^2} - \frac{\varepsilon}{2} \int_{\Omega} v \frac{|\nabla w|^2}{w^2}\notag\\& + \frac{1}{2}(\mu_1+\frac{r}{2})\int_{\Omega} u^2 \ln w+ \frac{1}{2}(\mu_2+\frac{r}{2})  \int_{\Omega} v^2 \ln w\notag- \varepsilon \int_{\Omega} \frac{u^2}{w} -\varepsilon \int_{\Omega} \frac{v^2}{w}+C.   \end{align*}
\end{lemma}
\begin{proof}
Let \(l_1(t)=-\int_{\Omega} u \ln w\), \(l_2(t)= -\int_{\Omega} v \ln w \). By directly differentiating, followed by applying integration by parts and Young's inequality, we can get
\begin{align}\label{eq:3.5}
l_1'(t)&= -\int_{\Omega} u_t \ln w - \int_{\Omega} \frac{u}{w} w_t\notag\\
\quad&= -\int_{\Omega} \left( \Delta u - \chi_1 \nabla \cdot \left( u \frac{\nabla w}{w^k} \right) + w - \mu_1 u^2 \right) \ln w- \int_{\Omega} \frac{u}{w} (\Delta w +u+v- w )\notag\\
\quad&= 2 \int_{\Omega} \nabla u \cdot \frac{\nabla w}{w} - \chi_1 \int_{\Omega} \frac{u}{w^{1+k}} |\nabla w|^2-\int_{\Omega} w \ln w + \mu_1 \int_{\Omega} u^2 \ln w\notag\\&\quad- \int_{\Omega} \frac{u}{w^2} |\nabla w|^2 - \int_{\Omega} \frac{u^2}{w}-\int_{\Omega} \frac{uv}{w}+ \int_{\Omega} u   \notag\\
\quad& \leq 2 \int_{\Omega} \frac{|\nabla u|^2}{u}  
- \frac{1}{2} \int_{\Omega} u \frac{|\nabla w|^2}{w^2} -\int_{\Omega} w\ln w  +\mu_1 \int_{\Omega} u^2 \ln w- \int_{\Omega} \frac{u^2}{w}+M.  
\end{align}
Using the same method, we can attain
\begin{align}\label{eq:3.6}
l_2'(t)&= -\int_{\Omega} v_t \ln w - \int_{\Omega} \frac{v}{w} w_t\notag\\
\quad&= -\int_{\Omega} \left( \Delta v - \chi_2 \nabla \cdot \left( v \frac{\nabla w}{w^k} \right) + w+ruv - \mu_2 v^2 \right) \ln w- \int_{\Omega} \frac{v}{w} (\Delta w +u+v- w )\notag\\
\quad&= 2 \int_{\Omega} \nabla v \cdot \frac{\nabla w}{w} - \chi_2 \int_{\Omega} \frac{v}{w^{1+k}} |\nabla w|^2-\int_{\Omega} w \ln w -r\int_{\Omega}uv\ln w+ \mu_2 \int_{\Omega} v^2 \ln w\notag\\&\quad- \int_{\Omega} \frac{v}{w^2} |\nabla w|^2  - \int_{\Omega} \frac{v^2}{w} -\int_{\Omega} \frac{uv}{w}+ \int_{\Omega} v \notag\\
\quad &\leq  2 \int_{\Omega} \frac{|\nabla v|^2}{v}  
- \frac{1}{2} \int_{\Omega}v \frac{|\nabla w|^2}{w^2}  -\int_{\Omega} w \ln w+\frac{r}{2}\int_{\Omega} u^2 \ln w+\frac{r}{2}\int_{\Omega} v^2 \ln w+ \mu_2 \int_{\Omega} v^2 \ln w- \int_{\Omega} \frac{v^2}{w}+M. 
\end{align} 
Combing \eqref{eq:3.5} and \eqref{eq:3.6}, we have
\begin{align}\label{eq:3.7}
\varepsilon l'(t) + \varepsilon l(t) \leq  & 2 \varepsilon \int_{\Omega} \frac{|\nabla u|^2}{u}+ 2\varepsilon \int_{\Omega} \frac{|\nabla v|^2}{v}- \frac{\varepsilon}{2} \int_{\Omega} u \frac{|\nabla w|^2}{w^2} - \frac{\varepsilon}{2} \int_{\Omega} v \frac{|\nabla w|^2}{w^2}- 2\varepsilon \int_{\Omega} w \ln w\notag\\& + (\mu_1+\frac{r}{2}) \varepsilon \int_{\Omega} u^2 \ln w+ (\mu_2+\frac{r}{2}) \varepsilon \int_{\Omega} v^2 \ln w\notag- \varepsilon \int_{\Omega} \frac{u^2}{w} -\varepsilon \int_{\Omega} \frac{v^2}{w}\notag\\&- \varepsilon \int_{\Omega} u \ln w - \varepsilon \int_{\Omega} v \ln w+2\varepsilon M .     
\end{align}
By using an elementary inequality that \( xy \leq x \ln x + e^{y-1} \) for all \( x > 0 \) and \( y > 0 \) and Lemma \ref{lem:3.2}, we can derive that
\begin{align}\label{eq:3.8}
(\mu_1+\frac{r}{2}) \varepsilon \int_{\Omega} u^2 \ln w &= \frac{1}{4}(\mu_1+\frac{r}{2}) \int_{\Omega} u^2 \ln w^{4\varepsilon}\notag\\
&\leq \frac{1}{2}(\mu_1+\frac{r}{2}) \int_{\Omega} u^2 \ln u +\frac{1}{4} (\mu_1+\frac{r}{2}) \int_{\Omega} w^{4\varepsilon}\notag\\
&\leq \frac{1}{2}(\mu_1+\frac{r}{2}) \int_{\Omega} u^2 \ln u + c_1,   
\end{align}
where \(\varepsilon \geq \frac{1}{4}\). Applying Young's inequality and Lemma \ref{lem:3.2} entails that
\begin{align}\label{eq:3.9}
-\varepsilon\int_{\Omega} u \ln w \leq  \varepsilon\int_{\Omega} \frac{u^2}{w} + \frac{\varepsilon}{4} \int_{\Omega} w \ln^2 w
\leq \varepsilon \int_{\Omega} \frac{u^2}{w} + c_2.   
\end{align}
By employing the same methodology, we can infer that
\begin{align}\label{eq:3.10}
(\mu_2+\frac{r}{2}) \varepsilon \int_{\Omega} v^2 \ln w
\leq \frac{1}{2}(\mu_2+\frac{r}{2}) \int_{\Omega} v^2 \ln v + c_1,    
\end{align}
\begin{align}\label{eq:3.11}
-\varepsilon\int_{\Omega} v \ln w \leq  \varepsilon \int_{\Omega} \frac{v^2}{w} + \frac{ \varepsilon}{4} \int_{\Omega} w \ln^2 w
\leq \varepsilon \int_{\Omega} \frac{v^2}{w} + c_2,    
\end{align}
\begin{align}\label{eq:3.12}
 -2\varepsilon \int_{\Omega} w \ln w\leq c_3,
\end{align}
where  $c_1$, $c_2$, $c_3 > 0$ . 
By combining  \eqref{eq:3.7}-\eqref{eq:3.12}, and utilizing Lemma \ref{lem:3.1}, we have 
\begin{align}
\varepsilon l'(t) + \varepsilon l(t) \leq  & 2 \varepsilon \int_{\Omega} \frac{|\nabla u|^2}{u}+ 2\varepsilon \int_{\Omega} \frac{|\nabla v|^2}{v}- \frac{\varepsilon}{2} \int_{\Omega} u \frac{|\nabla w|^2}{w^2} - \frac{\varepsilon}{2} \int_{\Omega} v \frac{|\nabla w|^2}{w^2}\notag\\& + \frac{1}{2}(\mu_1+\frac{r}{2})\int_{\Omega} u^2 \ln w+ \frac{1}{2}(\mu_2+\frac{r}{2})  \int_{\Omega} v^2 \ln w\notag- \varepsilon \int_{\Omega} \frac{u^2}{w} -\varepsilon \int_{\Omega} \frac{v^2}{w}+2\varepsilon M.    
\end{align}
\end{proof}
We are now poised to establish the first pivotal component in proving the primary result, specifically, an \( L \ln L \) bound for  both \( u \) and \( v \).
\begin{lemma}\label{lem:3.5}
There exists \( C > 0 \) such that
\begin{align*}
\int_{\Omega} u(\cdot, t) \ln u(\cdot, t)+\int_{\Omega} v(\cdot, t) \ln v(\cdot, t) \leq C  \quad \text{for all } t \in (0, T_{\text{max}}).
\end{align*}
\end{lemma}
\begin{proof}
Let \( y(t) := \int_{\Omega} u(\cdot, t) \ln u(\cdot, t) +\int_{\Omega} v(\cdot, t) \ln v(\cdot, t) + \frac{1}{2} \int_{\Omega} |\nabla w(\cdot, t)|^2 \).
By differentiating the function with respect to time, we can acquire
\begin{align*}
y'(t) = \int_{\Omega} (\ln u + 1) u_t +\int_{\Omega} (\ln v + 1) v_t+ \int_{\Omega} \nabla w \cdot \nabla w_t
:= I_1 +I_2+ I_3.    
\end{align*}
Making use of the first equation of \eqref{eq:1.3} and integration by parts implies that
\begin{align}\label{eq:3.13}
I_1 &= \int_{\Omega} (\ln u + 1)\left( \Delta u - \chi_1 \nabla \cdot \left(  \frac{u}{w^k}\nabla w \right) + w - \mu_1 u^2 \right)\notag\\
&= -\int_{\Omega} \frac{|\nabla u|^2}{u} + \chi_1 \int_{\Omega} \frac{\nabla u \cdot \nabla w}{w^k} + \int_{\Omega} (w- \mu_1 u^2)(\ln u + 1).      
\end{align}
By applying Young's inequality with \( \varepsilon > 0 \), which will be determined later, we can obtain
\begin{align}\label{eq:3.14}
\chi_1 \int_{\Omega} \frac{\nabla u \cdot \nabla w}{w^k}
&\leq \varepsilon \int_{\Omega} \frac{|\nabla u|^2}{u} + \frac{\chi_1^2}{4\varepsilon} \int_{\Omega} u \frac{|\nabla w|^2}{w^{2k}}\notag\\
&\leq \varepsilon \int_{\Omega} \frac{|\nabla u|^2}{u} + \frac{\varepsilon}{2} \int_{\Omega} u \frac{|\nabla w|^2}{w^2} + c_1 \int_{\Omega} u |\nabla w|^2\notag\\
&\leq \varepsilon \int_{\Omega} \frac{|\nabla u|^2}{u} + \frac{\varepsilon}{2} \int_{\Omega} u \frac{|\nabla w|^2}{w^2} + \varepsilon \int_{\Omega} |\nabla w|^4 + c_2 \int_{\Omega} u^2,  \end{align}
where \( c_1 > 0 \) and \( c_2 > 0 \). Applying Gagliardo–Nirenberg interpolation inequality and Lemma \ref{lem:3.3} deduces that
\begin{align}\label{eq:3.15}
\varepsilon \int_{\Omega} |\nabla w|^4 
&\leq c_3 \varepsilon \int_{\Omega} (\Delta w)^2 \int_{\Omega} |\nabla w|^2 + c_3 \varepsilon \left( \int_{\Omega} |\nabla w|^2 \right)^2\notag\\&\leq  c_4\varepsilon \int_{\Omega} (\Delta w)^2 + c_5,    
\end{align}
where \( c_3, c_4, \) and \( c_5 \) are positive constants. By integrating \eqref{eq:3.13}-\eqref{eq:3.15}, we can conclude that
\begin{align}\label{eq:3.16}
\frac{d}{dt} \int_{\Omega} u\ln u + \int_{\Omega}u\ln u
&\leq (\varepsilon-1)\int_{\Omega} \frac{|\nabla u|^2}{u} +  \frac{\varepsilon}{2}\int_{\Omega} u\frac{|\nabla w|^2}{w^2} +\varepsilon c_4 \int_{\Omega} (\Delta w)^2 + c_2 \int_{\Omega} u^2  \notag\\&\quad+ \int_{\Omega} (w- \mu_1 u^2)(\ln u + 1)+ \int_{\Omega} u \ln u+c_5.      
\end{align}
Through the application of the same technique, we have
\begin{align*}
I_2 &= \int_{\Omega} (\ln v + 1)\left( \Delta v - \chi_2 \nabla \cdot \left(  \frac{v}{w^k}\nabla w \right) + w +ruv- \mu_2 v^2 \right)\notag\\
&= -\int_{\Omega} \frac{|\nabla v|^2}{v} + \chi_2 \int_{\Omega} \frac{\nabla v \cdot \nabla w}{w^k} + \int_{\Omega} (w+ruv- \mu_2 v^2)(\ln v + 1).      
\end{align*}
Then, it is straightforward to obtain
\begin{align}\label{eq:3.17}
\frac{d}{dt} \int_{\Omega} v\ln v + \int_{\Omega}v\ln v
&\leq (\varepsilon-1)\int_{\Omega} \frac{|\nabla v|^2}{v} +  \frac{\varepsilon}{2}\int_{\Omega} v\frac{|\nabla w|^2}{w^2} +\varepsilon c_4 \int_{\Omega} (\Delta w)^2 + c_2 \int_{\Omega} v^2  \notag\\&\quad+ \int_{\Omega} (w+ruv- \mu_2 v^2)(\ln v + 1)+ \int_{\Omega} v \ln v+c_5.   
\end{align}
Using integration by parts, Young’s inequality and Lemma \ref{lem:3.3} implies that
\begin{align}\label{eq:3.18}
I_3 + \frac{1}{2} \int_{\Omega} |\nabla w|^2 &=-\int_{\Omega}\Delta w (\Delta w+u+v-w)+\frac{1}{2} \int_{\Omega} |\nabla w|^2 \notag\\&=-\int_{\Omega} (\Delta w)^2 - \int_{\Omega} u \Delta w- \int_{\Omega} v \Delta w +\int_{\Omega}w\Delta w+\frac{1}{2}\int_{\Omega} |\nabla w|^2 \notag\\  
&\leq -\frac{1}{2} \int_{\Omega} (\Delta w)^2 +\int_{\Omega} u^2+\int_{\Omega} v^2.
\end{align}
One can verify that
\begin{align}\label{eq:3.19}
&\left( c_2 +1 \right) \int_{\Omega} u^2 +\left( c_2 +1 \right) \int_{\Omega} v^2+ \int_{\Omega} u \ln u+ \int_{\Omega} v\ln v \notag\\&+ \int_{\Omega} (w - \mu_1 u^2)(\ln u + 1) + \int_{\Omega} (w+ruv - \mu_2 v^2)(\ln v + 1)\notag\\  
\leq& -\frac{1}{2}(\mu_1+\frac{r}{2}) \int_{\Omega} u^2\ln u -\frac{1}{2}(\mu_2+\frac{r}{2}) \int_{\Omega} v^2\ln v+ c_6    
\end{align}
for some \(c_6 > 0\). Based on\eqref{eq:3.17}-\eqref{eq:3.19}, we find that
\begin{align}\label{eq:3.20}
y'(t) + y(t) &\leq (\varepsilon - 1) \int_{\Omega} \frac{|\nabla u|^2}{u} + \frac{\varepsilon}{2} \int_{\Omega} u \frac{|\nabla w|^2}{w^2} +(\varepsilon - 1) \int_{\Omega} \frac{|\nabla v|^2}{v} + \frac{\varepsilon}{2} \int_{\Omega}v\frac{|\nabla w|^2}{w^2}
\notag\\&\quad + \left(2 \varepsilon c_4 - \frac{1}{2} \right) \int_{\Omega} (\Delta w)^2
-\frac{1}{2}(\mu_1+\frac{r}{2}) \int_{\Omega} u^2 \ln u-\frac{1}{2}(\mu_2+\frac{r}{2}) \int_{\Omega} v^2 \ln v + c_7,      
\end{align}
where \( c_7 = 2c_5 + c_6  \). 
Using lemma \ref{lem:3.4}, we have 
\begin{align}\label{eq:3.21}
y'(t) + y(t) + \varepsilon l'(t) + \varepsilon l(t) \leq &(-1 + 3\varepsilon) \int_{\Omega} \frac{|\nabla u|^2}{u} + (-1 + 3\varepsilon) \int_{\Omega} \frac{|\nabla v|^2}{v} \notag\\&+\left(2\varepsilon c_4 - \frac{1}{2}\right) \int_{\Omega} (\Delta w)^2 + c_8.    
\end{align}
Choosing \( \varepsilon = \min \left\{ \frac{1}{3}, \frac{1}{4c_4} \right\} \) deduces that
\begin{align*}
y'(t) + y(t) + \varepsilon l'(t) + \varepsilon l(t) \leq c_9.  \end{align*}
Applying Gronwall's inequality to this entails that
\begin{align*}
\int_{\Omega} u \ln u +\int_{\Omega} v \ln v - \varepsilon \int_{\Omega} u \ln w  - \varepsilon \int_{\Omega} v \ln w + \frac{1}{2} \int_{\Omega} |\nabla w|^2 \leq c_{10}
\end{align*}
for some \( c_{10} > 0 \). This, together with the inequality that \( xy \leq x \ln x + e^{y-1} \) for all \( x > 0 \) and \( y > 0 \) implies that
\begin{align*}
\int_{\Omega} u \ln u+\int_{\Omega} v \ln v &\leq \varepsilon \int_{\Omega} u \ln w +\varepsilon \int_{\Omega} v \ln w+c_{10} \notag\\
&\leq \varepsilon \int_{\Omega} u \ln u + \frac{\varepsilon}{e} \int_{\Omega} w + \varepsilon \int_{\Omega} v \ln v+ \frac{\varepsilon} {e} \int_{\Omega} w+c_{10}\notag\\
&\leq \varepsilon \int_{\Omega} u \ln u +\varepsilon \int_{\Omega} v \ln v+ c_{11}.    
\end{align*}
Note that \( \varepsilon < 1 \); thus, the lemma follows.
\end{proof}
Next, we proceed to derive the subsequent estimate concerning the gradient of \( w \). The following lemma can be proven through a standard testing argument, as exemplified in \cite{le2024global} [Lemma 4.2].  This estimate will play a crucial role in facilitating theestablishment of an \( L^p \) bound for both \( u \) and \( v \).

\begin{lemma}\label{lem:3.6}
For any \(p > 1\), there exist positive constants \(C_1\), \(C_2\), \(C_3\) depending only on \(p\) such that 
\begin{align*}
\frac{d}{dt} \int_{\Omega} |\nabla w|^{2p} + \int_{\Omega} |\nabla w|^{2p} &\leq -C_1 \int_{\Omega} \left| \nabla |\nabla w|^p \right|^2 + C_2 \int_{\Omega} u^2 |\nabla w|^{2p - 2} \\&\quad+C_2 \int_{\Omega} v^2 |\nabla w|^{2p - 2} + C_3 \int_{\Omega} |\nabla w|^{2p}.   
\end{align*}
\end{lemma} 
\begin{proof}
We make use of the following pointwise identity
\[
\nabla w \cdot \nabla \Delta w = \frac{1}{2} \Delta \left( |\nabla w|^2 \right) - |D^2 w|^2
\]
to obtain
\begin{align}\label{eq:3.22}
\frac{1}{2p} \frac{d}{dt} \int_\Omega |\nabla w|^{2p} + \int_\Omega |\nabla w|^{2p} 
=&-c_1 \int_\Omega |\nabla|\nabla w|^p|^2 - \int_\Omega |\nabla w|^{2p-2} |D^2 w|^2 + \int_\Omega |\nabla w|^{2p-2} \nabla w \cdot \nabla u \notag\\
&+ \int_\Omega |\nabla w|^{2p-2} \nabla w \cdot \nabla v + c_2 \int_{\partial \Omega} \frac{\partial |\nabla w|^2}{\partial \nu} |\nabla w|^{2p-2},     
\end{align}
where \( c_1, c_2 \) are positive constants depending only on \( p \). The inequality (\text{see} \cite{mizoguchi2014nondegeneracy}[Lemma 4.2])
\[\frac{\partial |\nabla w|^2}{\partial \nu} \leq M |\nabla w|^2  \]
for some \( M > 0 \) depending only on \( \Omega \), implies that
\[
c_2 \int_{\partial \Omega} \frac{\partial |\nabla w|^2}{\partial \nu} |\nabla w|^{2p-2} \, dS 
\leq c_2 M \int_{\partial \Omega} |\nabla w|^{2p} \, dS.
\]
Let \( f := |\nabla w|^p \) and apply the trace embedding theorem \( W^{1,1}(\Omega) \hookrightarrow L^1(\partial \Omega) \) together with Young's inequality; there exist positive constants \( c_3 \) and \( c_4\) such that
\[
c_2 M \int_{\partial \Omega} f^2 \, dS 
\leq c_3 \int_\Omega f |\nabla f| + c_3 \int_\Omega f^2 
\leq \frac{c_1}{2} \int_\Omega  |\nabla f|^2 + c_4 \int_\Omega f^2.
\]
Therefore, we have
\[
c_2 M \int_{\partial \Omega} |\nabla w|^{2p} \, dS 
\leq \frac{c_1}{2} \int_\Omega |\nabla|\nabla w|^p|^2 + c_4 \int_\Omega |\nabla w|^{2p}.
\]
Applying the pointwise inequality \( (\Delta w)^2 \leq 2 |D^2 w|^2 \) to \eqref{eq:3.22} yields
\begin{align}\label{eq:3.23}
\frac{1}{2p} \frac{d}{dt} \int_\Omega |\nabla w|^{2p} + \int_\Omega |\nabla w|^{2p}
\leq& -\frac{c_1}{2}\int_\Omega |\nabla|\nabla w|^p|^2
- \frac{1}{2} \int_\Omega |\nabla w|^{2p-2} |\Delta  w|^2
+ \int_\Omega |\nabla w|^{2p-2} \nabla w \cdot \nabla u\notag \\& + \int_\Omega |\nabla w|^{2p-2} \nabla w \cdot \nabla v
+ c_4 \int_\Omega |\nabla w|^{2p}.   
\end{align}
By integration by parts and Youngs's inequalities, there exist constants \( c_5  > 0 \) and \( c_6  > 0 \) such that
\begin{align}\label{eq:3.24}
\int_\Omega |\nabla w|^{2p-2} \nabla w \cdot \nabla u \ 
&= -\int_\Omega u |\nabla w|^{2p-2} \Delta w \ - \frac{2p-2}{p} \int_\Omega u |\nabla w|^{p-1}\nabla|\nabla w|^p\cdot \frac{\nabla w}{|\nabla w|} \notag\\
&\leq \frac{1}{4} \int_\Omega (\Delta w)^2 |\nabla w|^{2p-2} + \frac{c_1}{5} \int_\Omega \left| \nabla |\nabla w|^p \right|^2 + c_6 \int_\Omega u^2 |\nabla w|^{2p-2},
\end{align}
\begin{align}\label{eq:3.25}
\int_\Omega |\nabla w|^{2p-2} \nabla w \cdot \nabla v \ 
&= -\int_\Omega v |\nabla w|^{2p-2} \Delta w \ -\frac{2p-2}{p}  \int_\Omega v |\nabla w|^{p-1} \nabla|\nabla w|^p \cdot\frac{\nabla w}{|\nabla w|} \notag\\
&\leq \frac{1}{4} \int_\Omega (\Delta w)^2 |\nabla w|^{2p-2} + \frac{c_1}{5} \int_\Omega \left| \nabla |\nabla w|^p \right|^2  + c_6 \int_\Omega v^2 |\nabla w|^{2p-2}.
\end{align}
From \eqref{eq:3.23}-\eqref{eq:3.25}, we finally prove this lemma.
\end{proof}
Now, we move on to derive the second key component necessary for proving the main result, as outlined in the lemma below.
\begin{lemma}\label{lem:3.7}
For \( p > 1 \) and \( 0 < q < p - 1 \), there exists \( C > 0 \) such that
\begin{align*}
\int_{\Omega} u^p +\int_{\Omega} v^p + \int_{\Omega} u^pw^{-q}+\int_{\Omega} v^pw^{-q}+\int_{\Omega} |\nabla w|^{2p}\leq C\quad \text{for all } t \in (0, T_{\text{max}}).   
\end{align*} 
\end{lemma}
\begin{proof}
Differentiating the following functional and substituting \eqref{eq:1.3} yields
\begin{align*}
&\frac{d}{dt} \int_{\Omega} u^pw^{-q} +\frac{d}{dt} \int_{\Omega} v^pw^{-q} \\=& p \int_{\Omega} u^{p-1}w^{-q}u_t - q \int_{\Omega} u^pw^{-q-1}w_t +p \int_{\Omega} v^{p-1}w^{-q}v_t - q \int_{\Omega} v^pw^{-q-1}w_t\\
=&p \int_{\Omega} u^{p-1}w^{-q} \left(\Delta u - \chi_1 \nabla \cdot \left( \frac{u}{w^k} \nabla w \right) + w - \mu_1 u^2\right)- q \int_{\Omega} u^pw^{-q-1} (\Delta w +u+ v -w)\\&+p\int_{\Omega} v^{p-1}w^{-q}\left(\Delta v-\chi_2 \nabla \cdot \left(\frac{v}{w^k} \nabla w \right) + w +ruv-\mu_2 v^2 \right)-q \int_{\Omega} v^pw^{-q-1}(\Delta w+u+v-w). 
\end{align*}
By applying integration by parts, we can derive that
\begin{align*}
&\frac{d}{dt} \int_{\Omega} u^pw^{-q} +\frac{d}{dt} \int_{\Omega} v^pw^{-q} \\
= &-p(p-1) \int_{\Omega} u^{p-2}w^{-q}|\nabla u|^2 + 2pq \int_{\Omega} u^{p-1}w^{-q-1}\nabla u \cdot \nabla w + p(p-1)\chi_1 \int_{\Omega} u^{p-1}w^{-q-k}\nabla u \cdot \nabla w \\
& - pq\chi_1 \int_{\Omega} u^pw^{-q-k-1}|\nabla w|^2-q(q+1) \int_{\Omega} u^{p}w^{-q-2}|\nabla w|^2+ p\int_{\Omega} u^{p-1}w^{-q+1} - \mu_1 p \int_{\Omega} u^{p+1}w^{-q}\\& - q\int_{\Omega} u^{p+1}w^{-q-1}-q\int_{\Omega}u^pw^{-q-1}v+q\int_{\Omega}u^pw^{-q}-p(p-1) \int_{\Omega} v^{p-2}w^{-q}|\nabla v|^2 \\&+2pq \int_{\Omega} v^{p-1}w^{-q-1}\nabla v \cdot \nabla w  + p(p-1)\chi_2 \int_{\Omega} v^{p-1}w^{-q-k}\nabla v \cdot \nabla w - pq\chi_2 \int_{\Omega} v^pw^{-q-k-1}|\nabla w|^2 \\&- q(q+1) \int_{\Omega} v^{p}w^{-q-2}|\nabla w|^2 + p\int_{\Omega} v^{p-1}w^{-q+1} +rp\int_{\Omega}uv^pw^{-q}- \mu_2 p \int_{\Omega} v^{p+1}w^{-q}- q\int_{\Omega} v^{p+1}w^{-q-1}\\&-q\int_{\Omega}uv^pw^{-q-1}v+q\int_{\Omega}v^pw^{-q}
\end{align*}
and
\begin{align*}
&\frac{d}{dt} \int_{\Omega} u^p+\frac{d}{dt} \int_{\Omega} v^p \\=&p\int_{\Omega}u^{p-1}\left(\Delta u - \chi_1 \nabla \cdot \left( \frac{u}{w^k} \nabla w \right) + w - \mu_1 u^2\right)+p\int_{\Omega}v^{p-1}\left(\Delta v-\chi_2 \nabla \cdot \left(\frac{v}{w^k} \nabla w \right)+w+ruv-\mu_2 v^2\right )\\
=&-p(p-1)\int_{\Omega}u^{p-2}|\nabla u|^2+p(p-1)\chi_1\int_{\Omega}u^{p-1}w^{-k}\nabla u \cdot \nabla w+p\int_{\Omega}u^{p-1}(w-\mu_1 u^2)\\&-p(p-1)\int_{\Omega}v^{p-2}
|\nabla v|^2+p(p-1)\chi_2\int_{\Omega}v^{p-1}w^{-k}\nabla v \cdot \nabla w+p\int_{\Omega}v^{p-1}(w+ruv-\mu_2 v^2)\\
= &-\frac{4(p-1)}{p} \int_{\Omega} |\nabla u^{\frac{p}{2}}|^2 + 2\chi_1(p-1) \int_{\Omega} u^{\frac{p}{2}}w^{-k}\nabla u^{\frac{p}{2}} \cdot \nabla w +p \int_{\Omega} u^{p-1}w  - \mu_1 p \int_{\Omega} u^{p+1}\\& -\frac{4(p-1)}{p} \int_{\Omega} |\nabla v^{\frac{p}{2}}|^2 + 2\chi_2(p-1) \int_{\Omega} v^{\frac{p}{2}}w^{-k}\nabla v^{\frac{p}{2}} \cdot \nabla w  +p \int_{\Omega} v^{p-1}w+rp\int_{\Omega}uv^p- \mu_2 p \int_{\Omega} v^{p+1}. 
\end{align*}
Applying Young’s inequality with \(\varepsilon_1>0\), which will be determined later, implies that
\begin{align}\label{eq:3.26}
2pq\int_{\Omega} u^{p-1}w^{-q-1}\nabla u \cdot \nabla w 
\leq \bigl(p(p-1) - \varepsilon_1\bigr) \int_{\Omega} u^{p-2}w^{-q}|\nabla u|^2 + \frac{p^2q^2}{p(p-1) -\varepsilon_1} \int_{\Omega} u^{p}w^{-q-2}|\nabla w|^2,
\end{align}
\begin{align}\label{eq:3.27}
p(p-1)\chi_1 \int_{\Omega} u^{p-1}w^{-q-k}\nabla u \cdot \nabla w
&\leq \varepsilon_1 \int_{\Omega} u^{p-2}w^{-q}|\nabla u|^2 
+ \frac{p^2(p-1)^2\chi_1^2}{4\varepsilon_1} \int_{\Omega} u^{p}w^{-q-2k}|\nabla w|^2  \notag\\
&\leq \varepsilon_1 \int_{\Omega} u^{p-2}w^{-q}|\nabla u|^2 
+ \varepsilon_1 \int_{\Omega} u^{p}w^{-q-2}|\nabla w|^2 
+ c_1 \int_{\Omega} u^p|\nabla w|^2,
\end{align}
\begin{align}\label{eq:3.28}
2\chi_1(p-1)\int_{\Omega} u^{\frac{p}{2}}w^{-k}\nabla u^{\frac{p}{2}} \cdot \nabla w 
&\leq \frac{p-1}{p} \int_{\Omega} |\nabla u^{\frac{p}{2}}|^2 
+ p(p-1)\chi_1^2 \int_{\Omega} u^{p}w^{-2k}|\nabla w|^2 \notag\\
&\leq \frac{p-1}{p} \int_{\Omega} |\nabla u^{\frac{p}{2}}|^2 
+ \varepsilon_1 \int_{\Omega} u^{p}w^{-q-2}|\nabla w|^2 
+ c_2 \int_{\Omega} u^p|\nabla w|^2. 
\end{align}
Utilizing the same operational procedure, we see that
\begin{align}\label{eq:3.29}
2pq\int_{\Omega} v^{p-1}w^{-q-1}\nabla v \cdot \nabla w 
&\leq \bigl(p(p-1) - \varepsilon_1\bigr) \int_{\Omega} v^{p-2}w^{-q}|\nabla w|^2 + \frac{p^2q^2}{p(p-1) - \varepsilon_1} \int_{\Omega} v^{p}w^{-q-2}|\nabla w|^2,
\end{align}
\begin{align}\label{eq:3.30}
p(p-1)\chi_2 \int_{\Omega} v^{p-1}w^{-q-k}\nabla v \cdot \nabla w
&\leq \varepsilon_1 \int_{\Omega} v^{p-2}w^{-q}|\nabla v|^2 
+ \varepsilon_1 \int_{\Omega} v^{p}w^{-q-2}|\nabla w|^2 
+ c_1 \int_{\Omega} v^p|\nabla w|^2,
\end{align}
\begin{align}\label{eq:3.31}
2\chi_2(p-1)\int_{\Omega} v^{\frac{p}{2}}w^{-k}\nabla v^{\frac{p}{2}} \cdot \nabla w 
\leq \frac{p-1}{p} \int_{\Omega} |\nabla v^{\frac{p}{2}}|^2 
+ \varepsilon_1 \int_{\Omega} v^{p}w^{-q-2}|\nabla w|^2 
+ c_2 \int_{\Omega} v^p|\nabla w|^2,    
\end{align}
where \(c_1  > 0\) and \(c_2  > 0\).
Using Young’s inequality again and Lemma \ref{lem:3.2} entails that
\begin{align}\label{eq:3.32}
p \int_{\Omega} u^{p-1} w^{-q+1} &\leq \frac{q}{2} \int_{\Omega} u^{p+1} w^{-q-1} + c_3 \int_{\Omega} w^{p-q} \leq \frac{q}{2} \int_{\Omega} u^{p+1} w^{-q-1} + c_4,  
\end{align}
\begin{align}\label{eq:3.33}
(q+1) \int_{\Omega} u^p w^{-q} &\leq \frac{q}{2} \int_{\Omega} u^{p+1} w^{-q-1} + c_5 \int_{\Omega} w^{p-q} \leq \frac{q}{2} \int_{\Omega} u^{p+1} w^{-q-1} + c_6, 
\end{align}
\begin{align}\label{eq:3.34}
p \int_{\Omega} u^{p-1} w &\leq \frac{\mu_1 p}{3} \int_{\Omega} u^{p+1} + c_7 \int_{\Omega} w^{\frac{p+1}{2}} \leq \frac{\mu_1 p}{3} \int_{\Omega} u^{p+1} + c_8,  
\end{align}
\begin{align}\label{eq:3.35}
\int_{\Omega} u^p &\leq \frac{\mu_1 p}{3} \int_{\Omega} u^{p+1} + c_9.     
\end{align}
Similarly, by repeating the procedure, we can easily get
\begin{align}\label{eq:3.36}
p \int_{\Omega} v^{p-1} w^{-q+1} &\leq \frac{q}{2} \int_{\Omega} v^{p+1} w^{-q-1} + c_3 \int_{\Omega} w^{p-q} \leq \frac{q}{2} \int_{\Omega} v^{p+1} w^{-q-1} + c_4, 
 \end{align}
\begin{align}\label{eq:3.37}
 (q+1) \int_{\Omega} v^p w^{-q}\leq \frac{q}{2} \int_{\Omega} v^{p+1} w^{-q-1} + c_6,    
\end{align}
\begin{align}\label{eq:3.38}
p \int_{\Omega} v^{p-1} w  \leq \frac{\mu_2 p}{3} \int_{\Omega} v^{p+1} + c_8., 
\end{align}
\begin{align}\label{eq:3.39}
\int_{\Omega} v^p &\leq \frac{\mu_2 p}{3} \int_{\Omega} v^{p+1} + c_9.     
\end{align}
Applying Young’s inequality, we can figure out
\begin{align}\label{eq:3.40}
rp \int_{\Omega} uv^{p}w^{-q} &\leq \frac{\mu_1 p}{2} \int_{\Omega} u^{p+1} w^{-q} + \frac{\mu_2 p}{2} \int_{\Omega} v^{p+1} w^{-q},
\end{align}
\begin{align}\label{eq:3.41}
  rp \int_{\Omega} uv^p &\leq \frac{\mu_1 p}{3} \int_{\Omega} u^{p+1} + \frac{\mu_1 p}{3} \int_{\Omega} v^{p+1}.  
\end{align}
Collecting from \eqref{eq:3.26}-\eqref{eq:3.41} leads to
\begin{align}\label{eq:3.42}
&\frac{d}{dt} \int_{\Omega} u^p w^{-q} + \int_{\Omega} u^p w^{-q} +\frac{d}{dt} \int_{\Omega} v^p w^{-q} + \int_{\Omega} v^p w^{-q} + \frac{d}{dt} \int_{\Omega} u^p + \int_{\Omega} u^p+\frac{d}{dt} \int_{\Omega} v^p + \int_{\Omega} v^p\notag\\ 
\leq& \left( \frac{p^2 q^2}{p(p-1) - \varepsilon_1} + 2\varepsilon_1 - q(q+1) \right) \int_{\Omega} u^p w^{-q-2} |\nabla w|^2- \frac{3(p-1)}{p} \int_{\Omega} |\nabla u^{\frac{p}{2}}|^2\notag\\& +\left( \frac{p^2 q^2}{p(p-1) - \varepsilon_1} + 2\varepsilon_1 - q(q+1) \right) \int_{\Omega} v^p w^{-q-2} |\nabla w|^2- \frac{3(p-1)}{p} \int_{\Omega} |\nabla v^{\frac{p}{2}}|^2\notag\\& + c_{10} \int_{\Omega} u^p |\nabla w|^2 + c_{10} \int_{\Omega} v^p |\nabla w|^2  +c_{11}, 
\end{align}
where \( c_{10} =c_1 + c_2 \) and \( c_{11} = 2(c_4 + c_6 + c_8 + c_9) \). The condition \( 0 < q < p - 1 \) entails that \( \frac{p^2 q^2}{p(p-1)} < q(q+1) \), which allows us to choose \( \varepsilon_1 \) sufficiently small such that
\begin{align*}
\frac{p^2 q^2}{p(p-1) - \varepsilon_1} + 2\varepsilon_1 - q(q+1) \leq 0.   
\end{align*}
This, together with \eqref{eq:3.42} infers that
\begin{align*}
&\frac{d}{dt} \int_{\Omega} u^p w^{-q} + \int_{\Omega} u^p w^{-q} +\frac{d}{dt} \int_{\Omega} v^p w^{-q} + \int_{\Omega} v^p w^{-q} + \frac{d}{dt} \int_{\Omega} u^p + \int_{\Omega} u^p+\frac{d}{dt} \int_{\Omega} v^p + \int_{\Omega} v^p\\
\leq& -\frac{3(p-1)}{p} \int_{\Omega} |\nabla u^{\frac{p}{2}}|^2  -\frac{3(p-1)}{p} \int_{\Omega} |\nabla v^{\frac{p}{2}}|^2+ c_{10} \int_{\Omega} u^p |\nabla w|^2 +c_{10}\int_{\Omega} v^p |\nabla w|^2+ c_{11}.    
\end{align*}
Setting \( y(t) := \int_{\Omega} u^p w^{-q} + \int_{\Omega} v^p w^{-q}+\int_{\Omega} u^p +\int_{\Omega} v^p + \int_{\Omega} |\nabla w|^{2p} \) and applying Lemma \ref{lem:3.6} deduces that
\begin{align}\label{eq:3.43}
y'(t) + y(t) \leq& -\frac{3(p-1)}{p} \int_{\Omega} |\nabla u^{\frac{p}{2}}|^2 -\frac{3(p-1)}{p} \int_{\Omega} |\nabla v^{\frac{p}{2}}|^2- c_{12} \int_{\Omega} |\nabla| \nabla w|^p|^2 + c_{10} \int_{\Omega} u^p |\nabla w|^2\notag\\&+c_{10}\int_{\Omega} v^p |\nabla w|^2+ c_{13} \int_{\Omega} u^2 |\nabla w|^{2p-2} + c_{13} \int_{\Omega} v^2 |\nabla w|^{2p-2}+ c_{14} \int_{\Omega} |\nabla w|^{2p} + c_{11},     
\end{align}
where $c_{12}$, $c_{13}$, $c_{14}$ are positive constants depending only on \(p\). In light of Young’s inequality, we can obtain that
\begin{align}\label{eq:3.44}
&c_{10} \int_{\Omega} u^p |\nabla w|^2 +c_{10} \int_{\Omega} v^p |\nabla w|^2+ c_{13} \int_{\Omega} u^2 |\nabla w|^{2p-2} + c_{13} \int_{\Omega} v^2 |\nabla w|^{2p-2} + c_{14} \int_{\Omega} |\nabla w|^{2p} + c_{11}\notag\\
\leq& \varepsilon_2 \int_{\Omega} |\nabla w|^{2p+2} + c_{15} \int_{\Omega} u^{p+1} +c_{15} \int_{\Omega} v^{p+1} + c_{16},   
\end{align}
where \(\varepsilon_2 >0\) will be determined later, \(c_{15} > 0\) and \(c_{16}> 0\). Applying Gagliardo–Nirenberg interpolation inequality and Lemma \ref{lem:3.3} yields that
\begin{equation}
\begin{split}\label{eq:3.45}
\varepsilon_2 \int_{\Omega} |\nabla w|^{2p+2} &\leq \varepsilon_2 c_{17} \int_{\Omega} \bigl| \nabla |\nabla w|^p \bigr|^2 \int_{\Omega} |\nabla w|^2 + \varepsilon_2 c_{17} \left( \int_{\Omega} |\nabla w|^2 \right)^{\frac{p}{2}}\\
&\leq c_{12} \int_{\Omega} \bigl| \nabla |\nabla w|^p \bigr|^2 + c_{18},    
\end{split} 
\end{equation}
where \(c_{17} > 0\),  \(c_{18} > 0\) and 
\begin{align*}
\varepsilon_2 = \frac{c_{12}}{c_{17} \sup_{t \in (0, T_{\text{max}})} \int_{\Omega} |\nabla w(\cdot, t)|^2}.    
\end{align*}
Then, Lemma \ref{lem:3.5} asserts that
\[\sup_{t \in (0, T_{\text{max}})} \int_{\Omega} u \ln u < \infty  \quad and    \sup_{t \in (0, T_{\text{max}})} \int_{\Omega} v \ln v < \infty,\]
which allows us to apply Lemma \ref{lem:2.4} with
\begin{align*}
 \varepsilon_3 = \frac{p-1}{p c_{15} \sup_{t \in (0, T_{\text{max}})} \int_{\Omega} u \ln(u+e)}\quad and\quad   \varepsilon_4 = \frac{p-1}{p c_{15} \sup_{t \in (0, T_{\text{max}})} \int_{\Omega} v \ln(v+e)} 
\end{align*}
to obtain that
\begin{align}\label{eq:3.46}
c_{15} \int_{\Omega} u^{p+1} &\leq \varepsilon_3 c_{15} \int_{\Omega} |\nabla u^{\frac{p}{2}}|^2 \int_{\Omega} u \ln(u+e) + \varepsilon_3 c_{15} \left( \int_{\Omega} u \right)^p \int_{\Omega} u \ln(u+e) + c_{19}\notag\\
&\leq \frac{p-1}{p} \int_{\Omega} |\nabla u^{\frac{p}{2}}|^2 + c_{20},   
\end{align}
\begin{align}\label{eq:3.47}
c_{15} \int_{\Omega} v^{p+1} &\leq \varepsilon_4 c_{15} \int_{\Omega} |\nabla v^{\frac{p}{2}}|^2 \int_{\Omega} v \ln(v+e) + \varepsilon_4 c_{15} \left( \int_{\Omega} v \right)^p \int_{\Omega} v \ln(v+e) + c_{19}\notag\\
&\leq \frac{p-1}{p} \int_{\Omega} |\nabla v^{\frac{p}{2}}|^2 + c_{20},    
\end{align}
where \(c_{19} > 0\) and \(c_{20} > 0\). Collecting from (\ref{eq:3.43})-(\ref{eq:3.47}) implies that
\[
y'(t) + y(t) \leq c_{21}  \quad \text{for all } t \in (0, T_{\text{max}}),
\]where \(c_{21} = c_{11} + c_{16} + c_{18} + 2c_{20}\). Applying Gronwall's inequality to this entails that
\[
y(t) \leq \max \left\{ \int u_0^p+\int v_0^p + \int u_0^p w_0^{-q} + \int v_0^p w_0^{-q}+ \int |\nabla w_0|^{2p},\; c_{21} \right\} \quad \text{for all } t \in (0, T_{\text{max}}),
\]
which completes the proof.
\end{proof}
Drawing upon the implications of the lemma presented above, we now conclude this section by introducing the subsequent estimate.
\begin{lemma}\label{lem:3.8}
Let \( p > \max\left\{2, \frac{1}{1-k}\right\} \). Then there exists \( C > 0 \) such that
\[
\int_{\Omega} \frac{u^p(\cdot, t) \, |\nabla w(\cdot, t)|^p}{w^{kp}(\cdot, t)} \leq C\quad  \text{and}\quad 
\int_{\Omega} \frac{v^p(\cdot, t) \, |\nabla w(\cdot, t)|^p}{w^{kp}(\cdot, t)} \leq C\quad \text{for all t}  \in (0, T_{\text{max}}). 
\]
\end{lemma}
\begin{proof}
From Lemma \ref{lem:3.7}, it follows that \( u,v \in L^\infty((0, T_{\text{max}})\); \(L^p(\Omega)) \) for some \( p > 2 \). This, in conjunction with  Minkowski's inequality \((||u+v||_{L^p (U)}\leq ||u||_{L^p (U)}+||v||_{L^p (U)}\) for \(1 \leq p \leq \infty\) and \(u\), \(v\in L^p (U)\)) and  Lemma \ref{lem:2.2} entails that
\[
\sup_{t \in (0, T_{\text{max}})} \|\nabla w(\cdot, t)\|_{L^\infty(\Omega)} = c_1 < \infty.
\]
Moreover, we notice that \( kp < p - 1 \) since \( p > \frac{1}{1-k} \), which together with Lemma \ref{lem:3.7} implies that
\[
\sup_{t \in (0, T_{\text{max}})} \int_\Omega \frac{u^p(\cdot, t)}{w^{kp}(\cdot, t)} = c_2 < \infty\quad \text{and}\quad 
\sup_{t \in (0, T_{\text{max}})} \int_\Omega \frac{v^p(\cdot, t)}{w^{kp}(\cdot, t)} = c_2 < \infty.
\]
Therefore, we can infer that
\[
\int_\Omega \frac{u^p(\cdot, t) |\nabla w(\cdot, t)|^p}{w^{kp}(\cdot, t)}
\leq \sup_{t \in (0, T_{\text{max}})} \|\nabla w(\cdot, t)\|_{L^\infty(\Omega)}^p \int_\Omega \frac{u^p(\cdot, t)}{w^{kp}(\cdot, t)}
\leq c_1^p c_2 \quad \text{for all } t \in (0, T_{\text{max}}),
\]
\[
\int_\Omega \frac{v^p(\cdot, t) |\nabla w(\cdot, t)|^p}{w^{kp}(\cdot, t)}
\leq \sup_{t \in (0, T_{\text{max}})} \|\nabla w(\cdot, t)\|_{L^\infty(\Omega)}^p \int_\Omega \frac{v^p(\cdot, t)}{w^{kp}(\cdot, t)}
\leq c_1^p c_2 \quad \text{for all } t \in (0, T_{\text{max}}),
\]
which proves the lemma.
\end{proof}
\begin{proof}[\textbf{Proof of Theorem \ref{thm:1.1}}] \quad 
By utilizing standard estimates pertaining to heat semigroups, we are now in a position to establish the uniform boundedness of solutions. We use a similar method in the proof of Theorem 0.1 in \cite{winkler2010boundedness}.
\begin{align}\label{eq:3.48}
u(\cdot, t) = e^{(t-t_0)\Delta}u(\cdot, t_0) + \int_{t_0}^t e^{(t-s)\Delta}\nabla \cdot \left( u\frac{\nabla w}{w^k} \right)(\cdot, s) \, ds + \int_{t_0}^t e^{(t-s)\Delta} (w-\mu_1u^2(\cdot, s)) \, ds, 
\end{align}
where \( t \in (0, T_{\text{max}}) \) and \( t_0 = \max \{ 0, t - 1 \} \). Applying standard \( L^p - L^q \) estimates for \( (e^{t\Delta})_{t \geq 0} \) (see \cite{winkler2010aggregation} [Lemma 1.3]) implies that
\begin{align}\label{eq:3.49}
\left\| e^{(t-t_0)\Delta}u(\cdot, t_0) \right\|_{L^\infty(\Omega)} \leq c_1 M(1 + (t - t_0)^{-1}),    
\end{align}
where \( c_1 > 0 \) and \( M \) is the constant in Lemma \ref{lem:3.1}. For \( q > \max \left\{ 2, \frac{1}{1-k} \right\} \) and \( p \in (q, \infty) \), applying Lemma \ref{lem:3.8} entails that
\begin{align}\label{eq:3.50}
\left\|\int_{t_0}^t e^{(t-s)\Delta}\nabla \cdot \left( u\frac{\nabla w}{w^k} \right)(\cdot, s) \, ds\right\|_{L^\infty(\Omega)} &\leq \int_{t_0}^t \left\| e^{(t-s)\Delta}\nabla \cdot \left( u\frac{\nabla w}{w^k} \right)(\cdot, s) \right\|_{L^\infty(\Omega)} \, ds\notag\\   
&\leq c_2 \int_{t_0}^t (t - s)^{-\frac{1}{p}} \left\| e^{\frac{t-s}{2}\Delta}\nabla \cdot \left( u\frac{\nabla w}{w^k} \right)(\cdot, s) \right\|_{L^p(\Omega)} \, ds\notag\\
&\leq c_3 \int_{t_0}^t (t - s)^{-\frac{1}{2} - \frac{1}{q}} \left\| u(\cdot, s) \frac{\nabla w(\cdot, s)}{w^k(\cdot, s)} \right\|_{L^q(\Omega)} \, ds\notag\\&
\leq c_4 (t - t_0)^{\frac{1}{2} - \frac{1}{q}}.
\end{align}
In addition, by utilizing Lemma \ref{lem:3.2} and Lemma \ref{lem:3.7}, we have
\begin{align}\label{eq:3.51}
\left\|\int_{t_0}^t e^{(t-s)\Delta} (w-\mu_1u^2) \, ds\right\|_{L^\infty(\Omega)} 
\leq c_5(t - t_0).    
\end{align}
Collecting from \eqref{eq:3.48}-\eqref{eq:3.51} yields
\begin{align*}
\|u(\cdot, t)\|_{L^\infty(\Omega)} &\leq c_6\bigl((t - t_0)^{-1}  + (t - t_0)^{-\frac{1}{2} - \frac{1}{q}}+(t - t_0)\bigr)\\
&\leq c_7(t^{-1} + 1)  \quad \text{for all } t \in (0, T_{\max}).    
\end{align*}
On the other hand, we have 
\begin{align*}
c_8rv - \mu_2 v^2 \leq \frac{(c_8r)^2}{4\mu_2}\quad \text{for all}\quad t \in (0, T_{\max}). 
\end{align*} 
Then we can get \(\|v\|_{L^{\infty}(\Omega)}\leq c_9\) by the same way, where \( c_9 > 0\). Consequently, one can get
\begin{align*}
\|u(\cdot, t)\|_{L^\infty(\Omega)} + \|v(\cdot, t)\|_{L^\infty(\Omega)} + \|w(\cdot, t)\|_{W^{1,\infty}(\Omega)} \leq C \quad \text{for all} \quad t \in (0, T_{\max}),  
\end{align*}
which further deduces that \(u\) is bounded in \(\Omega\times (0,T_{\text{max}})\) since \(u\) is bounded in  \(\Omega\times(0,\frac{T_{\text{max}} }{2})\)
 by
Lemma \ref{lem:2.1}. Therefore, by the extensibility property of the solutions as established in Lemma \ref{lem:2.1},
it follows that \(T_{\text{max}} =\infty\) and that \((u,v,w)\) is bounded in \(\Omega\times(0,\infty)\).
\end{proof}
\section{ Asymptoptic behavior}
\hspace*{\parindent}
In this section, based on the above established uniform boundedness, we prove that the positive global solutions of \eqref{eq:1.3} exponentially converge to the positive equilibrium.

\vspace{\baselineskip} 
\noindent\textit{4.1. Stabilization}
\vspace{\baselineskip} 

We construct a Lyapunov function, the form of which is inspired by \cite{lou2021role}. 
\begin{lemma}\label{lem:4.1}
Let the assumptions of Theorem \ref{thm:1.1} hold and let \( a > 0 \) be given in \eqref{eq:1.6}, \( K > 0 \). Assume that \( \mu_1, \mu_2 \) and \( r \) fulfill
\begin{align}\label{eq:4.1}
 \mu_1 < \mu_2 < 3\mu_1   
\end{align}
and
\begin{align}
r = \mu_2 - \mu_1.  \label{eq:4.2}   
\end{align}
Then there exists \( \chi_0 > 0 \) such that whenever \( \chi_1^2 + \chi_2^2 \leq \chi_0 \) and \( (u, v, w) \) is a global classical solution of the boundary value problem in \eqref{eq:1.3} such that \( u, v \) and \( w \) are positive in \( \Omega \times (0, \infty) \) with
\begin{align}\label{eq:4.3}
\|w(\cdot, t)\|_{L^\infty(\Omega)} \leq K \quad \text{for all } t > 0.   
\end{align}
For
\begin{align*}
\mathcal{F}(t) &:= \int_\Omega \left\{ u(\cdot, t) - u_* - u_* \ln \frac{u(\cdot, t)}{u_*} \right\} + \int_\Omega \left\{ v(\cdot, t) - v_* - v_* \ln \frac{v(\cdot, t)}{v_*} \right\}
\\&\quad+ 2 \int_\Omega \left\{ w(\cdot, t) - w_* - w_* \ln \frac{w(\cdot, t)}{w_*} \right\}  \quad \text{for all } t > 0
\end{align*}
and
\begin{align*}
\mathcal{E}(t) &:= \int_\Omega (u - u_*)^2 +  \int_\Omega (v - v_*)^2 \quad \text{for all } t > 0,   
\end{align*}
we can find some constants \( \varepsilon > 0 \) such that
\begin{align}\label{eq:4.4}
\frac{d}{dt} \mathcal{F}(t) \leq -\varepsilon\mathcal{E}(t) \quad \text{for all } t > 0,    
\end{align}
where \( (u_*, v_*, w_*) \) is defined in \eqref{eq:1.5}.
\end{lemma}
\begin{proof}
Given \( K> 0 \), we let \( \chi_0:= \frac{16}{1+K^2} \) and suppose that \( \chi_1^2 + \chi_2^2 \leq \chi_0 \), and \( (u, v, w) \) is a global classical solution of the boundary value problem in \eqref{eq:1.3} fulfilling \( u, v, w > 0 \) in \( \bar{\Omega} \times (0, \infty) \) as well with \eqref{eq:4.3}. Then, straightforward calculations entail that
\begin{align}\label{eq:4.5}
\frac{d}{dt} \int_{\Omega} \left( u - u_* - u_* \ln \frac{u}{u_*} \right) &= \int_{\Omega} \left( 1 - \frac{u_*}{u} \right) u_t \notag\\  
&= -u_* \int_{\Omega} \frac{|\nabla u|^2}{u^2} + \chi_1 u_* \int_{\Omega}  \frac{\nabla u }{u}\cdot \frac{\nabla w}{w^k} + \int_{\Omega} \left( 1 - \frac{u_*}{u} \right) (w - \mu_1 u^2)
\end{align}
for all \( t > 0 \). We apply Young's inequality to estimate
\begin{align}\label{eq:4.6}
\chi_1 u_* \int_{\Omega} \frac{\nabla u}{u} \cdot \frac{\nabla w }{w^k}\leq u_* \int_{\Omega} \frac{|\nabla u|^2}{u^2} + \frac{u_* \chi_1^2}{4} \int_{\Omega} \frac{|\nabla w|^2 }{w^{2k}}\quad \text{for all } t > 0.    
\end{align}
Combing \eqref{eq:4.5} and \eqref{eq:4.6}, we have
\begin{align}\label{eq:4.7}
\frac{d}{dt} \int_{\Omega} \left( u - u_* - u_* \ln \frac{u}{u_*} \right) \leq \frac{u_* \chi_1^2}{4} \int_{\Omega} \frac{|\nabla w|^2}{w^{2k}} + \int_{\Omega} w - \mu_1 \int_{\Omega} u^2 - u_* \int_{\Omega} \frac{w}{u} + \mu_1 u_* \int_{\Omega} u   
\end{align}
for all \( t > 0 \). Similarly, it follows from the second and the third equations in \eqref{eq:1.3} that
\begin{align}\label{eq:4.8}
\frac{d}{dt} \int_{\Omega} \left( v - v_* - v_* \ln \frac{v}{v_*} \right) &\leq \frac{v_* \chi_2^2}{4} \int_{\Omega} \frac{|\nabla w|^2}{w^{2k}} + \int_{\Omega} w + r \int_{\Omega} uv - \mu_2 \int_{\Omega} v^2\notag\\&\quad
- v_* \int_{\Omega} \frac{w}{v} - rv_* \int_{\Omega} u + \mu_2 v_* \int_{\Omega} v \quad \text{for all } t > 0   
\end{align}
and
\begin{align}\label{eq:4.9}
\frac{d}{dt} \int_{\Omega} \left( w - w_* - w_* \ln \frac{w}{w_*} \right) 
&= -w_* \int_{\Omega} \frac{|\nabla w|^2}{w^2} + \int_{\Omega} u + \int_{\Omega} v - \int_{\Omega} w- w_* \int_{\Omega} \frac{u}{w} \notag\\&\quad - w_* \int_{\Omega} \frac{v}{w} + w_* \cdot |\Omega| \quad \text{for all } t > 0.   
\end{align}
Therefore, we infer from \eqref{eq:4.7}-\eqref{eq:4.9} that
\begin{align}\label{eq:4.10}
\frac{d}{dt} \mathcal{F}(t) &\leq \int_{\Omega} \left\{\frac{u_* \chi_1^2 + v_* \chi_2^2 }{4w^{2k}}- \frac{2w_*}{w^2}\right\} |\nabla w|^2\notag\\&\quad
-\mu_1 \int_{\Omega} (u - u_*)^2 + (2 - \mu_1 u_* - rv_*) \int_{\Omega} u + \mu_1 u_*^2 \cdot |\Omega|\notag\\&\quad
-\mu_2 \int_{\Omega} (v - v_*)^2 + (2 - \mu_2 v_*) \int_{\Omega} v + \mu_2 v_*^2 \cdot |\Omega|\notag\\&\quad
+ r \int_{\Omega} (u - u_*)(v - v_*) + rv_* \int_{\Omega} u + ru_* \int_{\Omega} v - ru_* v_* \cdot |\Omega|\notag\\&\quad
- u_* \int_{\Omega} \frac{w}{u} - 2w_* \int_{\Omega} \frac{u}{w} - v_* \int_{\Omega} \frac{w}{v} - 2w_* \int_{\Omega} \frac{v}{w} + 2w_* \cdot |\Omega| \quad \text{for all } t > 0.    
\end{align}
Since \(k\in(0,1)\), we can use interpolation to obtain a value between two known points. The specific steps are as follows
\begin{align}\label{eq:4.11}
\int_{\Omega} \left\{\frac{u_* \chi_1^2 + v_* \chi_2^2 }{4w^{2k}}- \frac{2w_*}{w^2}\right\} |\nabla w|^2 \leq \int_{\Omega} \left\{\frac{u_* \chi_1^2 + v_* \chi_2^2 }{4w^{2}}+\frac{u_* \chi_1^2 + v_* \chi_2^2 }{4}- \frac{2w_*}{w^2}\right\} |\nabla w|^2.
\end{align}
In light of Young's inequality, we can conclude that
\begin{align}\label{eq:4.12}
r \int_{\Omega} (u - u_*)(v - v_*) \leq \frac{r}{2} \int_{\Omega} (u - u_*)^2 + \frac{r}{2} \int_{\Omega} (v - v_*)^2 \quad \text{for all } t > 0,   
\end{align}
\begin{align}\label{eq:4.13}
- u_* \int_{\Omega} \frac{w}{u} - 2w_* \int_{\Omega} \frac{u}{w} \leq -2 \int_{\Omega} \sqrt{u_*\frac{w}{u} \cdot 2w_* \frac{u}{w}} = -2\sqrt{2u_* w_*} \cdot |\Omega| \quad \text{for all } t > 0,   
\end{align}
\begin{align}\label{eq:4.14}
- v_* \int_{\Omega} \frac{w}{v} - 2w_* \int_{\Omega} \frac{v}{w} \leq -2 \int_{\Omega} \sqrt{v_*\frac{w}{v} \cdot 2w_* \frac{v}{w}} = -2\sqrt{2v_* w_*} \cdot |\Omega| \quad \text{for all } t > 0.    
\end{align}
Hence, together with \eqref{eq:4.10}–\eqref{eq:4.14}, one can see that for all \( t > 0 \),
\begin{align}\label{eq:4.15}
\frac{d}{dt} \mathcal{F}(t) &\leq \int_{\Omega} \left\{\frac{u_* \chi_1^2 + v_* \chi_2^2 }{4w^{2}}+\frac{u_* \chi_1^2 + v_* \chi_2^2 }{4}- \frac{2w_*}{w^2}\right\} |\nabla w|^2\notag\\&\quad
- \left( \mu_1 - \frac{r}{2} \right) \int_{\Omega} (u - u_*)^2 - \left( \mu_2 - \frac{r}{2} \right) \int_{\Omega} (v - v_*)^2\notag\\&\quad
+ (2 - \mu_1 u_*) \int_{\Omega} u + (2 + ru_* - \mu_2 v_*) \int_{\Omega} v\notag\\&\quad
+ \{ (\mu_1 u_*^2 + \mu_2 v_*^2 - ru_* v_*) - 2(\sqrt{2u_*} w_* + \sqrt{2v_*} w_* - w_*) \} \cdot |\Omega|.    
\end{align}
Then, it follows from the definition of \( a \) in \eqref{eq:1.6} and \eqref{eq:4.2} that
\begin{align}\label{eq:4.16}
a = \frac{\mu_2 - \mu_1 + \sqrt{(\mu_2 - \mu_1)^2 + 4\mu_1\mu_2}}{2\mu_2} = 1.    
\end{align} 
By combining \eqref{eq:1.6} and \eqref{eq:4.16}, we can deduce that
\begin{align}\label{eq:4.17}
u_* = v_* = \frac{2}{\mu_1} \quad \text{and} \quad w_* = \frac{4}{\mu_1}.   
\end{align}
Also, by direct calculations, one can see
\begin{align}\label{eq:4.18}
2 - \mu_1 u_* = 2 - (1 + a) = 0,    
\end{align}
\begin{align}\label{eq:4.19}
2 + ru_* - \mu_2 v_* = 0,     
\end{align}
\begin{align}\label{eq:4.20}
\mu_1 u_*^2 + \mu_2 v_*^2 - ru_* v_* = (\mu_1 + \mu_2 - r) u_*^2 = 2\mu_1 \cdot \frac{4}{\mu_1^2} = \frac{8}{\mu_1}   
\end{align}
and
\begin{align}\label{eq:4.21}
2(\sqrt{2u_*} w_* + \sqrt{2v_*} w_* - w_*) = 2\left(2\sqrt{2 \cdot \frac{2}{\mu_1} \cdot \frac{4}{\mu_1}} - \frac{4}{\mu_1}\right) = \frac{8}{\mu_1}.    
\end{align}
From \eqref{eq:4.3} and the condition \( \chi_1^2 + \chi_2^2 \leq \chi_0 \), we have
\begin{align}\label{eq:4.22}
\left\{\frac{u_* \chi_1^2 + v_* \chi_2^2 }{4w^{2}}+\frac{u_* \chi_1^2 + v_* \chi_2^2 }{4}- \frac{2w_*}{w^2}\right\}&=\left\{\frac{u_* \chi_1^2 + v_* \chi_2^2 }{4}- \frac{8w_*-(u_*\chi_1^2+v_*\chi_2^2)}{4w^2}\right\} \notag\\
&\quad\leq \frac{\chi_1^2+\chi_2^2}{2\mu_1}  -\frac{16-(\chi_1^2+\chi_2^2)}{2\mu_1 K^2} \notag\\&\quad\leq 0.     
\end{align}
Furthermore,
\begin{align}\label{eq:4.23}
\mu_1 - \frac{r}{2} = \frac{3\mu_1 - \mu_2}{2} > 0 \quad \text{and} \quad \mu_2 - \frac{r}{2} = \frac{\mu_1 + \mu_2}{2} > 0     
\end{align}
implies that \eqref{eq:4.1}. Accordingly, we combine \eqref{eq:4.18}-\eqref{eq:4.23} to get \eqref{eq:4.4}.
\end{proof}
Before proving Theorem \ref{thm:1.2}(1), we shall show the following key property.
\begin{lemma}\textup{\cite{lou2021role}[Lemma 6.3]}\label{lem:4.2}
Let \( n \geq 1 \), and \( (u, v, w) \in (C^{2,1}(\bar{\Omega} \times (0, \infty)))^3 \) be a classical solution of \eqref{eq:1.3} in \( \Omega \times (0, \infty) \) for which \( u \geq 0 \), \( v \geq 0 \) and \( w \geq 0 \). And for which there exists \( K > 0 \) such that
\begin{align*}
\|u(\cdot, t)\|_{L^\infty(\Omega)} + \|v(\cdot, t)\|_{L^\infty(\Omega)} \leq K \quad \text{for all } t > 0.     
\end{align*}
Then there exist \( \theta \in (0, 1) \) and \( C > 0 \) such that
\[
\|u\|_{C^{\theta, \frac{\theta}{2}}(\bar{\Omega} \times [t, t+1])} + \|v\|_{C^{\theta, \frac{\theta}{2}}(\bar{\Omega} \times[t, t+1])} + \|w\|_{C^{1+\theta, \frac{1+\theta}{2}}(\bar{\Omega} \times [t, t+1])} \leq C \quad \text{for all } t > 1. 
\]
\end{lemma}
\begin{proof}[\textbf{Proof of Theorem \ref{thm:1.2}(1)}]\quad
Since \(\mathcal{F}\) is nonnegative due to the fact that \(s - 1 - \ln s \geq 0\) for all \(s > 0\). And since \(\mathcal{F}(1)\) is finite by positivity of \(u(\cdot,1)\), \(v(\cdot,1)\) and \(w(\cdot,1)\) in \(\bar{\Omega}\), integrating (\ref{eq:4.4}) in time shows that
\[
 c_1 \int_1^\infty \int_{\Omega} (u - u_*)^2 +c_2 \int_1^\infty \int_{\Omega} (v- v_*)^2 < \infty, 
\]
which due to \(c_1, c_2 > 0\) by our assumption, implies that
\begin{align}\label{eq:4.24}
\int_1^\infty \int_{\Omega} (u - u_*)^2 < \infty \quad    
and\quad
\int_1^\infty \int_{\Omega} (v - v_*)^2 < \infty.     
\end{align}
As both \( u - u_* \) and \( v - v_* \) are uniformly continuous in \( \Omega \times (1, \infty) \) thanks to Lemma \ref{lem:4.2} by a contradiction argument (cf. \cite{tao2019boundedness}, for instance), it follows from \eqref{eq:4.24} that
\[
u(\cdot, t) \to u_* \quad \text{and} \quad v(\cdot, t) \to v_* \quad \text{in } L^\infty(\Omega) \quad \text{as } t \to \infty.
\]
Finally, testing the third equation in \eqref{eq:1.3} against \( w - w_* \), using Young's inequality and relying on (\ref{eq:4.24}), we can easily find that
\[
\int_1^\infty \int_\Omega (w - w_*)^2 < \infty,
\]
it implies that
\[
w(\cdot, t) \to w_* \quad \text{in } L^\infty(\Omega) \quad \text{as } t \to \infty.
\]
This completes the proof. 
\end{proof}
\noindent\textit{4.2. Convergence rates}
\vspace{\baselineskip} 

In this subsection, we shall show the convergence rates of solutions in \(L^\infty\)-norm. Inspired by \cite{zhang2025globalexp}, we present the following lemma: 
\begin{lemma}\textup{\cite{2019Jinhaiyang}[Lemma 3.6] }\label{lem:4.3}
Let \(\Omega \subset \mathbb{R}^2\) be a bounded domain with smooth boundary and \((u, v, w)\) be a nonnegative bounded classical solution of system \eqref{eq:1.3}. And \(\chi_i(w) \) (\( i = 1, 2 \)) satisfy \(\bf(H1)\) and \(\bf(H2)\). Then for any \(q > 1\), there exists a constant \(C > 0\) independent of \(t\) such that
\begin{align*}
\| \nabla u(\cdot, t) \|_{L^{2q}(\Omega)} + \| \nabla v(\cdot, t) \|_{L^{2q}(\Omega)} \leq C, \quad t > 1.
\end{align*}
\end{lemma} 

The following lemma is the key to proving Theorem \ref{thm:1.2}(2). We shall use a perfect idea in \cite{Winkler2016Equilibration} [Lemma 3.7] to prove this lemma.
\begin{lemma}\label{lem:4.4}
Let the assumptions in Lemma \textup{\ref{lem:4.1}} be true. Then there exist two positive constants \(C_1\) and \(C_2\), such that
\begin{align*}
\|u(\cdot, t) - u_*\|_{L^2} + \|v(\cdot, t) - v_*\|_{L^2} + \|w(\cdot, t) - w_*\|_{L^2} \leq C_1 e^{-C_2t}  
\end{align*}
holds for all \(t > 0\).
\end{lemma}
\begin{proof}
We introduce the function \(\varphi(\omega) := \omega- u_* \ln \omega\) for \(\omega > 0\) and use the L'Hôpital's rule to derive
\begin{align*}
\lim_{\omega \to u_*} \frac{\varphi(\omega) - \varphi(u_*)}{(\omega - u_*)^2} = \lim_{\omega \to u_*} \frac{\varphi'(\omega)}{2(\omega - u_*)} = \frac{1}{2u_*}.   
\end{align*}
Since Theorem \ref{thm:1.2}(1) shows that  \(\|u(\cdot, t) - u_*\|_{L^\infty} \to 0\) as \(t \to \infty\), then we can select \(t_1 > 0\) such that for all \(t > t_1\)
\begin{align}\label{eq:4.25}
\frac{1}{4u_*} \int_{\Omega} (u - u_*)^2\leq\int_{\Omega} \left( u - u_* - u_* \ln \frac{u}{u_*} \right) = \int_{\Omega} (\varphi(u) - \varphi(u_*)) \le \frac{1}{u_*} \int_{\Omega} (u - u_*)^2.    
\end{align}
Similarly, the fact \(\|v(\cdot, t) - v_*\|_{L^\infty} \to 0\) as \(t \to \infty\) implies that there exists \(t_2 > 0\) such that for all \(t > t_2\)
\begin{align}\label{eq:4.26}
\frac{1}{4v_*} \int_{\Omega} (v - v_*)^2 \leq \int_{\Omega} \left( v - v_* - v_* \ln \frac{v}{v_*} \right) \leq \frac{1}{v_*} \int_{\Omega} (v - v_*)^2.   
\end{align}
Let \(t_3 =\text{max}\{1, t_1, t_2\}\). According to the definitions of \( \mathcal{F}(t) \) and \( \mathcal{E}(t) \) in Lemma \ref{lem:4.1}, the right inequalities of \eqref{eq:4.25} and \eqref{eq:4.26} indicate that \( \mathcal{F}(t)\leq c_1\mathcal{E}(t) \) for some \( c_1 > 0 \) and any \( t > t_3 \). Combining this with the energy inequality \eqref{eq:4.4}, one has
\begin{align}\label{eq:4.27}
\frac{d}{dt} \mathcal{F}(t) \leq -\varepsilon \mathcal{E}(t) \leq -\frac{\varepsilon}{c_1} \mathcal{F}(t) \qquad \text{for all}\quad t > t_3.   \end{align}
It follows from Gronwall’s inequality to \eqref{eq:4.27} and the boundedness of \( \mathcal{F}(t_3) \) that there exist \( c_2 > 0 \) and \( c_3 > 0 \) fulfilling
\begin{align}
\mathcal{F}(t) \leq c_2 e^{-c_3 t} \qquad \text{for all}\quad t > t_3.   
\end{align}
which combined with the left inequality of \eqref{eq:4.25} and  \eqref{eq:4.26}  gives
\begin{align*}
\int_{\Omega} (u - u_*)^2 + \int_{\Omega} (v - v_*)^2 + \int_{\Omega} (w - w_*)^2 \leq c_4 e^{-c_3 t} \qquad \text{for all}\quad t > t_3.    
\end{align*}
\end{proof}
\begin{proof}[\textbf{Proof of Theorem \ref{thm:1.2}(2)}]\quad
Let \( q = 2 \) in Lemma \ref{lem:4.3} and by the boundedness of \( \|w\|_{W^{1,\infty}(\Omega)} \) in Theorem \ref{thm:1.1}, we see that
\begin{align*}
\|u(\cdot, t)\|_{W^{1,4}(\Omega)} + \|v(\cdot, t)\|_{W^{1,4}(\Omega)} + \|w(\cdot, t)\|_{W^{1,4}(\Omega)} \leq c_1 \qquad \text{for all}\quad  t > t_3,   
\end{align*}
where \(t_3\) defined in Lemma \eqref{lem:4.4}. Hence, applying the Gagliardo–Nirenberg inequality to \( u - u^* \) yields
\begin{align}\label{eq:4.29}
\|u - u_*\|_{L^\infty(\Omega)} \leq c_2 \left( \|\nabla u\|_{L^4(\Omega)}^{\frac{2}{3}} \|u - u_*\|_{L^2(\Omega)}^{\frac{1}{3}} + \|u - u_*\|_{L^2(\Omega)}^{1} \right) \leq c_3 \|u - u_*\|_{L^2(\Omega)}^{\frac{1}{3}}.    
\end{align}
Similarly, one can obtain
\begin{align}\label{eq:4.30}
\|v - v_*\|_{L^\infty(\Omega)} \leq c_4 \|v - v_*\|_{L^2(\Omega)}^{\frac{1}{3}} \quad \text{and} \quad \|w - w_*\|_{L^\infty(\Omega)} \leq c_5 \|w - w_*\|_{L^2(\Omega)}^{\frac{1}{3}}.   
\end{align}
The combination of \eqref{eq:4.29} and \eqref{eq:4.30} and Lemma \ref{lem:4.4} shows
\begin{align*}
\|u(\cdot, t) - u_*\|_{L^\infty(\Omega)} + \|v(\cdot, t) - v_*\|_{L^\infty(\Omega)} + \|w(\cdot, t) - w_*\|_{L^\infty(\Omega)} \leq c_6 e^{-\frac{c_2}{6}t},
\end{align*}
where \(c_2\) is given by Lemma \ref{lem:4.4}. Choosing \( \lambda = \frac{c_2}{6} \) for all \( t > t_3 \). Moreover, it is valid for all \( t > 0 \) by picking the constant \( c_6 > 0 \) sufficiently large. The proof of Theorem \ref{thm:1.2}(2) is completed. 
\end{proof}

\section{Numerical simulation}
\hspace*{\parindent}
To reveal the spatiotemporal dynamics of the AA system \eqref{eq:1.3}, we illustrate some of our theoretical results by numerical calculations. To this end, we restrict ourselves to the two‑dimensional rectangular domain \(\Omega = [0, 2\pi] \times [0, 2\pi]\) with a uniform Cartesian grid, following the idea of \cite{zhang2025globalexp}. The spatial step sizes are taken as \(\Delta x = \Delta y = 0.5\), yielding \(13\times13\) grid points. The spatial derivatives of system \eqref{eq:1.3} are discretized using second‑order finite differences: the Laplacian is approximated by the standard five‑point stencil, while the chemotaxis term \(-\nabla \cdot (u \chi(w) \nabla w)\) is treated by a finite‑volume method that first computes fluxes on cell faces and then assembles the divergence, thereby guaranteeing mass conservation. Neumann (zero‑flux) boundary conditions are imposed consistently for all variables.

Throughout our numerical experiments, we employ the explicit Euler method for time integration with a fixed time step \(\Delta t = 0.01\), a value that satisfies the stability condition of the diffusive part. At each time step, we simultaneously update the three unknowns $u$, $v$ and $w$ by evaluating all spatial operators (diffusion, chemotaxis, and reaction) from the current state. A positivity‑preserving cutoff \(\max(\cdot, 10^{-6})\) is applied after each update to avoid non‑physical negative densities. The initial condition is a constant state \((u_0\), \(v_0\), \(w_0\)) = (\(2.5\), \(2.5\), \(5.0\)) plus a small normally distributed perturbation \(\mathcal{N}\sim(0, 0.2^2)\), which mimics a state far from the positive equilibrium. To monitor the spatio‑temporal evolution, we record the solution $u$ (since the distributions of the three variables are similar) at the four time instants  \(t=10\), \(20\), \(60\), \(1000\) and \(t=10\), \(20\), \(60\), \(1600\). For visualization, a three‑dimensional surface plot is generated after a seven‑fold linear interpolation of the coarse grid data, while a global colormap is set by the 1st–99th percentiles of the first three snapshots to allow direct comparison between different times.

Case 1 : \(k=0.8\)\quad  \(\mu_1=0.8 \quad \mu_2=0.9  \quad r=0.1.\)\quad 
Then \(u_0=2.5 \quad v_0=2.5 \quad w_0=5.\) 

Case 2 : \(k=1\)  \quad   \(\mu_1=0.8 \quad \mu_2=0.9  \quad r=0.1.\)\quad 
Then \(u_0=2.5 \quad v_0=2.5 \quad w_0=5.\) \label{Case 2} 

\begin{figure}[H]
	\centering
	\includegraphics[width=\linewidth]{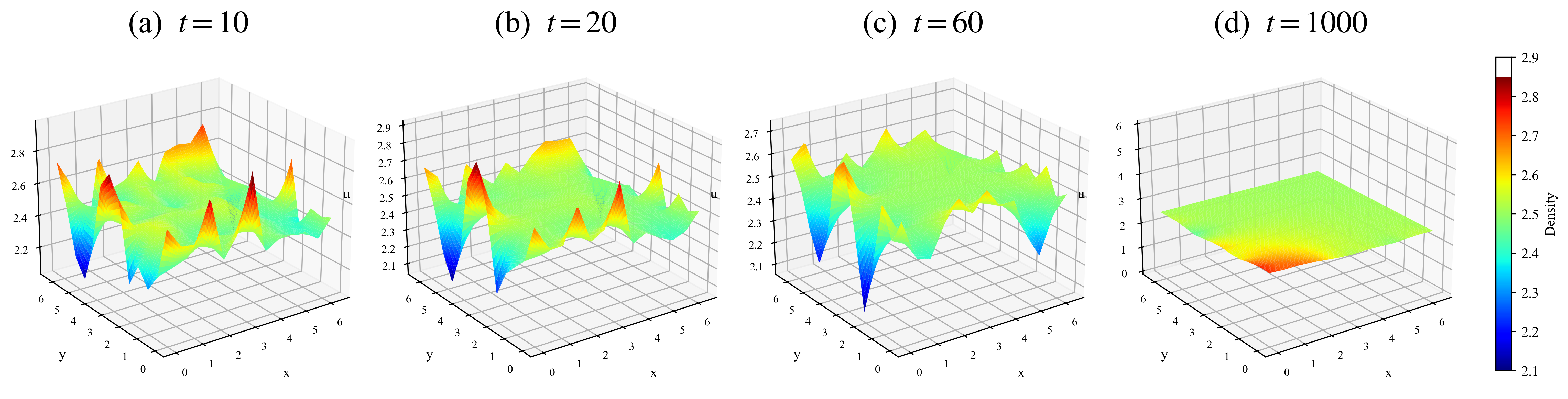}
    \caption{Spatio-temporal evolution and Turing instability of \(u(\cdot,x,y)\) for chemotaxis system (\ref{eq:1.3}) with parameter set Case 1 and \((u_0,v_0,w_0) = (2.5,2.5,5)+ \mathcal{N}\sim(0, 0.2^2)\) on \(\Omega = [0,2\pi]\times [0,2\pi]\).}
\end{figure}

\begin{figure}[H]
	\centering
	\includegraphics[width=\linewidth]{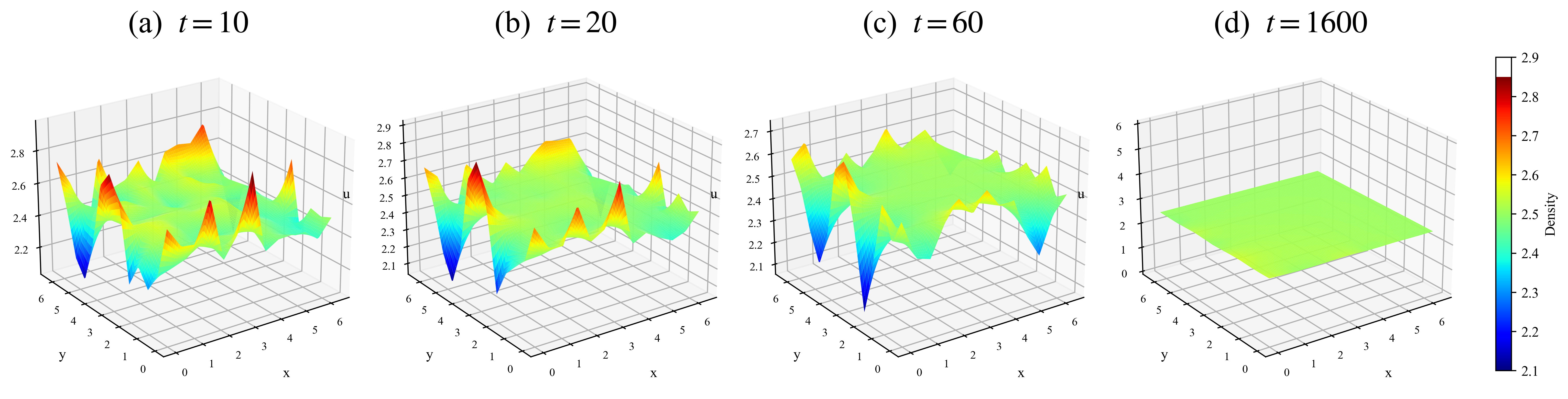}
    \caption{Spatio-temporal evolution and Turing instability of \(u(\cdot,x,y)\) for chemotaxis system (\ref{eq:1.3}) with parameter set Case 1 and \((u_0,v_0,w_0) = (2.5,2.5,5)+ \mathcal{N}\sim(0, 0.2^2)\) on \(\Omega = [0,2\pi]\times [0,2\pi]\).}
\end{figure}

\begin{figure}[H]
	\centering
	\includegraphics[width=\linewidth]{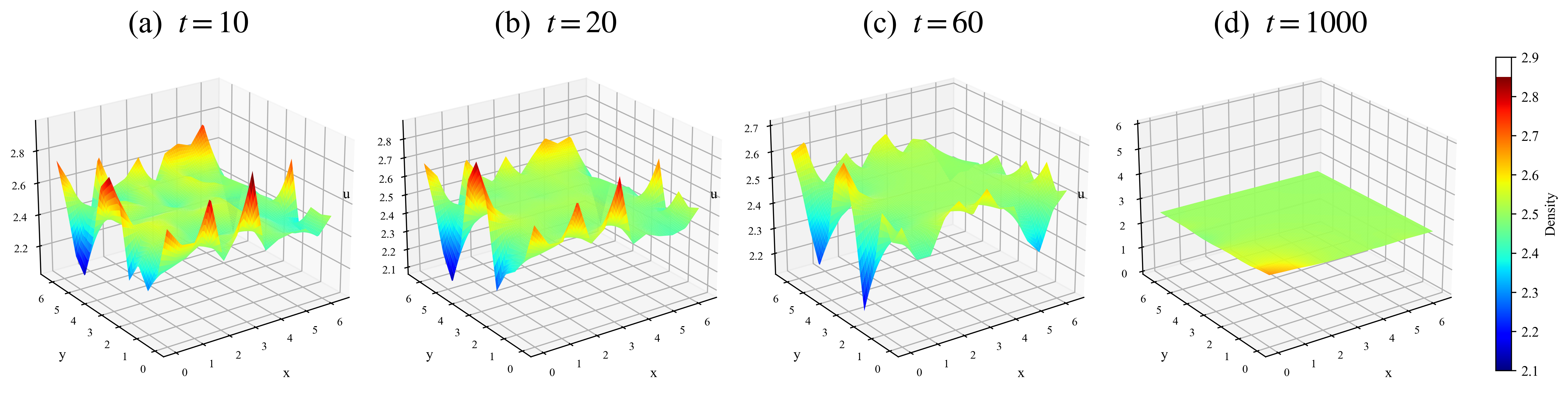}
    \caption{Spatio-temporal evolution and Turing instability of \(u(\cdot,x,y)\) for chemotaxis system (\ref{eq:1.3}) with parameter set Case 2  and \((u_0,v_0,w_0) = (2.5, 2.5,5)+ \mathcal{N}\sim(0, 0.2^2)\) on \(\Omega =[0,2\pi]\times [0,2\pi]\).}
\end{figure}
\begin{figure}[H]
	\centering
	\includegraphics[width=\linewidth]{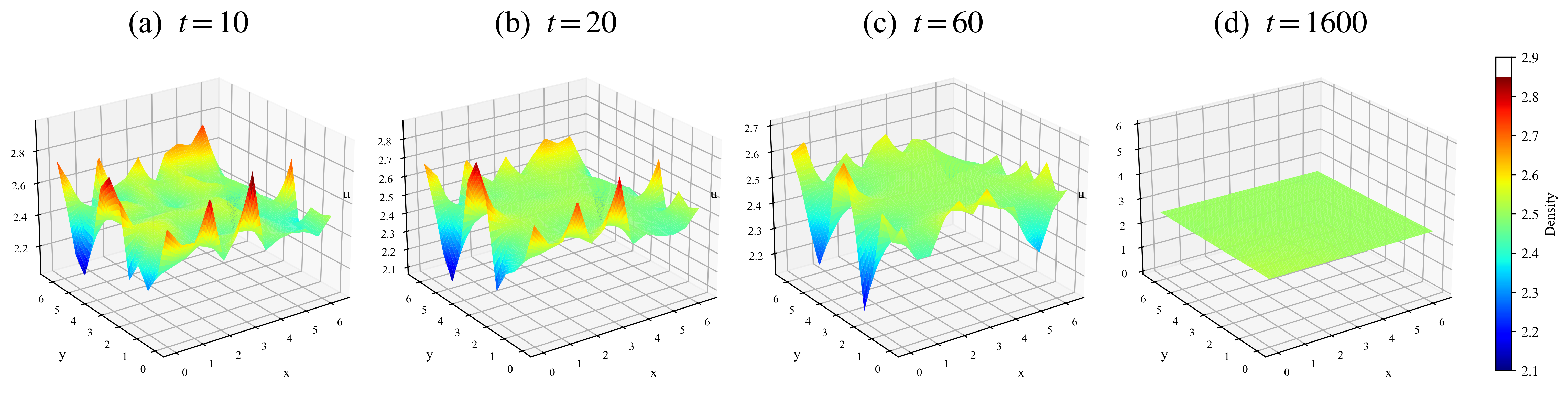}
     \caption{Spatio-temporal evolution and Turing instability of \(u(\cdot,x,y)\) for chemotaxis system (\ref{eq:1.3}) with parameter set Case 2  and \((u_0,v_0,w_0) = (2.5, 2.5,5)+ \mathcal{N}\sim(0, 0.2^2)\) on \(\Omega = [0,2\pi]\times [0,2\pi]\).}
\end{figure}

Considering a reaction-diffusion-coupling system, when initiated from initial random perturbations, its evolutionary process can be described through the following sequential stages: the waves diminish in intensity, followed by the merging of peaks and valleys, then a transition towards a quasi-uniform state, and finally the system attains an asymptotic steady state. Ultimately, the system trends towards a spatially uniform distribution (or an approximately planar distribution). At this stage, the combined net effects of reaction, diffusion, and coupling mechanisms equal zero, and the system achieves a homogeneous steady state. This behavior confirms the asymptotic stability of the system under specific parameter balancing and highlights the synergistic equilibrium between diffusion transport and nonlinear reaction.
Furthermore, when the chemotaxis index \(0<k<1\), the convergence rate is significantly slowed down, and the system requires a longer time to reach the steady state.

\section*{Acknowledgement}
\hspace*{\parindent}
The second author is supported by National Natural Science Foundation of China (No.12561039), and Corps Science and Technology Plan Project (No.2023CB008-13). The third author is supported by Launch Project of High-Level Talent Scientific Research of Shihezi University (No.RCZK202411), and Corps Science and Technology Plan Project ( No.2025DA049).
\vspace{\baselineskip} 
\section*{Conflict of Interest} 
\hspace*{\parindent}
The authors declare that this work does not have any conflicts of interest.

\end{document}